\documentclass[reqno, 12pt]{amsart}

\usepackage{style}

\title[Noncommutative kernels]{Kernels for noncommutative projective schemes}

\author[Ballard]{Matthew R. Ballard}
\address{Department of Mathematics, 
  University of South Carolina, 
  Columbia, SC 29208 USA}
\email{ballard@math.sc.edu}
\urladdr{\url{http://people.math.sc.edu/ballard/}}

\author[Farman]{Blake A. Farman}
\address{Department of Mathematics,
  University of South Carolina,
  Columbia, SC 29208 USA}
\email{farmanb@math.sc.edu}
\urladdr{\url{http://people.math.sc.edu/farmanb/}}

\keywords{noncommutative algebra, noncommutative projective schemes, derived categories, Fourier-Mukai transforms}

\thanks{Both authors were partially supported by a NSF Standard Grant DMS-1501813. The first author would like to thank the Institute for Advanced Study for providing an enriching and focused environment during his membership. The paper benefited from valuable comments by Pieter Belmans. Both authors are incredibly grateful to have known Ragnar-Olaf Buchweitz.}
\begin{document}

\begin{abstract}
  We give a noncommutative geometric description of the internal Hom dg-category in the homotopy category of dg-categories between two noncommutative projective schemes in the style of Artin-Zhang. As an immediate application, we give a noncommutative projective derived Morita statement along lines of Rickard and Orlov. 
\end{abstract}

\maketitle

\section{Introduction} \label{section: Introduction}

One of the most useful and pleasing statements in algebra is the Morita Theorem \cite{Morita}.

\begin{theorem}
  Let \(A\) and \(B\) be rings and assume we have an equivalence of module categories 
  \begin{displaymath}
    F : \Mod{A} \to \Mod{B}.
  \end{displaymath}
  Then there exists an \(A\)-\(B\) bimodule which is projective as an \(A\)-module and isomorphic to \(B\) as a \(B\)-module. Consequently, \(P \otimes_A -\) is an equivalence. 
\end{theorem}

For noncommutative rings, one can easily find non-isomorphic examples of Morita equivalent rings.
One such easy example is the equivalence between the modules over a ring \(A\) and the modules over its matrix ring, \(\op{M}_n(A)\), given by the bimodule \(A^n\).
For commutative rings, two rings are Morita equivalent if and only if they are isomorphic. One can try to slacken the relationship by considering the derived version of Morita for rings \cite{Rickard}.

\begin{theorem}
  Let \(A\) and \(B\) be rings and assume we have an equivalence of derived categories of modules
  \begin{displaymath}
    F : \mathrm{D}(\Mod{A}) \to \mathrm{D}(\Mod{B}).
  \end{displaymath}
  Then there exists a complex of \(A\)-\(B\) bimodules which is perfect as a complex of \(A\)-modules and whose derived endomorphisms as \(B\)-modules are just \(B\) in degree \(0\). Consequently, \(P \overset{\mathbf{L}}{\otimes}_A -\) is an equivalence. 
\end{theorem}

Even still, here two commutative rings are derived Morita equivalent if and only if they are isomorphic. To get a strictly weaker equivalence relation, we should globalize the notion of a commutative ring by passing to schemes. 

\begin{theorem}
  Let \(X\) and \(Y\) be quasi-compact, quasi-separated schemes over a field, \(k\), and assume we have an equivalence 
  \begin{displaymath}
    F: \mathrm{D}(\Qcoh{X}) \to \mathrm{D}(\Qcoh{Y}).
  \end{displaymath}
  Then there exists an object \(P \in \mathrm{D}(\Qcoh{X \times Y})\) such that
  \begin{displaymath}
    M \mapsto \Phi_P (M) := \mathbf{R}\pi_{Y \ast} \left( P \overset{\mathbf{L}}{\otimes} \mathbf{L}\pi_X^\ast M \right) 
  \end{displaymath}
  gives an equivalence. 
\end{theorem}

Here \(X \overset{\pi_X}\longleftarrow X \times Y \overset{\pi_Y}\longrightarrow Y\) are the projections. This result, as stated, follows immediately from results of T\"oen and Lunts-Orlov \cite{Toen,Lunts-Orlov}. A more specific version for smooth projective schemes predates this \cite{Orlov}. In this generality, we unlock a deep and subtle equivalence relationship. Indeed, understanding these Fourier-Mukai partnerships is a central problem in the field of derived categories in algebraic geometry. 
Given the richness of the globalized Morita statement for commutative rings, one is led to think about a globalized version for noncommutative rings. A framework for phrasing such a question is due to Artin and Zhang \cite{AZ} and goes by the name Noncommutative Projective Geometry. A graded Morita theorem for noncommutative rings is due to Zhang \cite{Zhang}.

Following the line of thought above, one wonders about derived Morita theory for noncommutative projective schemes \(X\) and \(Y\) which are associated to connected graded \(k\)-algebras \(A\) and \(B\).
As an application of the main result of this article, we have the following statement. One says that \(A\) and \(B\) form a delightful couple if they are both Ext-finite in the sense of \cite{VdB}, both are left and right Noetherian, and both satisfy \(\chi^\circ(R)\) for \(R=A,A^\opp\) for \(A\) and \(R = B, B^\opp\) for \(B\). One can think of this requirement as Serre vanishing for the twisting sheaves plus some finite-dimensionality over \(k\).

\begin{theorem}
  Let \(X\) and \(Y\) be noncommutative projective schemes associated to a delightful couple of over a field \(k\), both of which are generated in degree one.
  Then for any equivalence \(\mathrm{D}(\Qcoh{X}) \to \mathrm{D}(\Qcoh{Y})\), there exists an object \(P\) of \(\mathrm{D}(\Qcoh{X \times_k Y})\) whose associated integral transform is an equivalence of Fourier-Mukai type.
\end{theorem}

\noindent
The interested reader can see Corollary~\ref{corollary: NCP morita degree 1} for a more careful statement of this result.

This result is, as the notation suggests, an application of a more general result. Work of T\"oen provides a good bicategorical structure for dg-categories up to quasi-equivalence \cite{Toen} - an internal Hom. The construction of \(\RHomc{\mathcal C, \mathcal D}\) is abstract even if the dg-categories \(\mathcal C\) and \(\mathcal D\) arise geometrically (or noncommutative geometrically). To unleash the power of T\"oen's work, one needs to identify the internal Hom more ``internally.'' Indeed, this is done for schemes in \cite{Toen}, for higher derived stacks (using machinery of Lurie in place of T\"oen) in \cite{BFN}, and matrix factorizations in \cite{Dyckerhoff,PV,BFK}. The main result here is the identification of this internal Hom in Noncommutative Projective Geometry, see Theorem~\ref{theorem: derived morita for NCP}. 

This work is part of the second author's thesis \cite{Farman}, where the homological conditions making up tastiness will be sorted out more, the calculus of kernels will be more fully developed as will conditions for fully faithfulness, and more applications to noncommutative invariants in the style of Kontsevich will be studied.

\subsection{Conventions} We let \(k\) denote a field.
Often, for ease of notation, \(\C(X,Y)\) will be used to refer to the morphims, \(\op{Hom}_\C(X,Y)\), between objects \(X\) and \(Y\) of a category \(\C\). We shall also use an undecorated \(\op{Hom}\) again depending on the complexity of the notation. 
Whenever \(\C\) has a natural enrichment over a category, \(\mathcal{V}\), we will denote by \(\underline{\C}(X,Y)\) the \(\mathcal{V}\)-object of morphisms.
For example, the category of complexes of \(k\)-vector spaces, \(\Ch{k}\), can be endowed with the the structure of a \(\Ch{k}\)-enriched category using the hom total complex, \(\CH{k}(C,D) := \underline{C}(k)(C,D)\) which has in degree \(n\) the \(k\)-vector space
\[\CH{k}(C,D)^n = \prod_{m \in \Z}\Mod{k}\left(C^m, D^{m + n}\right)\]
and differential
\[d(f) = d_D \circ f + (-1)^{n+1} f \circ d_C.\]
It should be noted that \(Z^0(\CH{k}(C,D)) = \Ch{k}(C,D)\).

\section{Background on DG-Categories}\label{section:background on dgcats}
Recall that a {\bf dg-category}, \(\A\), over \(k\) is a category enriched over the category of chain complexes, \(\Ch{k}\), a {\bf dg-functor}, \(F \colon \A \ra \B\) is a \(\Ch{k}\)-enriched functor, a {\bf morphism of dg-functors of degree \(n\)}, \(\eta \colon F \ra G\), is a \(\Ch{k}\)-enriched natural transformation such that \(\eta(A) \in \B\left(FA, GA\right)^n\) for all objects \(A\) of \(\mathcal \A\), and a {\bf morphism of dg-functors} is a degree zero, closed morphism of dg-functors.
We will denote by \(\dgcat{k}\) the 2-category of small \(\Ch{k}\)-enriched categories, and by \(\DGCAT{k}(\A,\B)\) the dg-category of dg-functors from \(\A\) to \(\B\).

Recall also that for \(\A\) and \(\B\) small dg categories, we may define a dg-category \(\A \otimes \B\) with objects \(\op{ob}(\A) \times \op{ob}(\B)\) and morphisms
\[(\A \otimes \B)\left((X,Y), (X^\prime,Y^\prime)\right) = \A(X,X^\prime) \otimes_k \B(Y, Y^\prime).\]
It is well known that there is an isomorphism
\[\dgcat{k}(\A \otimes \B, \C) \cong \dgcat{k}(\A, \DGCAT{k}(\B, \C)),\]
endowing \(\dgcat{k}\) with the structure of a symmetric monoidal closed category.

For any dg-category, \(\A\), we denote by \(Z^0(\A)\) the category with objects those of \(\A\) and morphisms
\[Z^0(\A)(A_1, A_2) := Z^0(\A(A_1,A_2)).\]
By \(H^0(\A)\) we denote the category with objects those of \(\A\) and morphisms
\[H^0(\A)(A_1, A_2) := H^0(\A(A_1,A_2)).\]
Following \cite{CS}, we say that two objects \(A_1\), \(A_2\) of a dg-category, \(\A\), are {\bf dg-isomorphic} (respectively, {\bf homotopy equivalent}) if there is a morphism \(f \in Z^0(\A)(A_1, A_2)\) such that \(f\) (respectively, the image of \(f\) in \(H^0(\A)(A_1, A_2)\)) is an isomorphism.
In such a case, we say that \(f\) is a {\bf dg-isomorphism} (respectively, {\bf homotopy equivalence}).

\subsection{The Model Structure on DG-Categories}
We collect here some basic results on the model structure for \(\dgcat{k}\).
For any dg-functor \(F \colon \A \to \B\), we say that \(F\) is
\begin{enumerate}[(i)]
\item
  {\bf quasi-fully faithful} if for any two objects \(A_1\), \(A_2\) of \(\A\) the morphism
  \[F(A_1, A_2) \colon \A(A_1, \A_2) \to \B(FA_1, FA_2)\]
  is a quasi-isomorphism of chain complexes,
\item
  {\bf quasi-essentially surjective} if the induced functor \(H^0(F) \colon H^0(\A) \to H^0(\B)\) is essentially surjective,
\item
  a {\bf quasi-equivalence} if \(F\) is quasi-fully faithful and quasi-essentially surjective,
\item
  a {\bf fibration} if \(F\) satisfies the following two conditions:
  \begin{enumerate}[(a)]
  \item
    for all objects \(A_1, A_2\) of \(\A\), the morphism \(F(A_1,A_2)\) is a degree-wise surjective morphism of complexes, and
  \item
    for any object \(A\) of \(\A\) and any isomorphism \(\eta \in H^0(\B)(H^0(F)A, B)\), there exists an isomorphism \(\nu \in H^0(\C)(A,A^\prime)\) such that \(H^0(F)(\nu) = \eta\).
  \end{enumerate}
\end{enumerate}
In \cite{Tab05} it is shown that taking the class of fibrations defined above and the class of weak equivalences to be the quasi-equivalences, \(\dgcat{k}\) becomes a cofibrantly generated model category.
The localization of \(\dgcat{k}\) at the class of quasi-equivalences is the homotopy category, \(\Ho{\dgcat{k}}\).
We will denote by \([\A,\B]\) the morphisms of \(\Ho{\dgcat{k}}\).

A small dg-category \(\A\) is said to be {\bf h-projective} if for all objects \(A_1, A_2\) of \(\A\) and any acyclic complex, \(C\), every morphism of complexes \(\A(A_1, A_2) \to C\) is null-homotopic.
In \cite{CS}, it is shown that there exists an h-projective category \(A^{\op{hp}}\) quasi-equivalent to \(\A\) and, as a result, the localization of the full subcategory of \(\dgcat{k}\) of h-projective dg-categories at the class of quasi-equivalences is equivalent to \(\Ho{\dgcat{k}}\).
In particular, when \(k\) is a field, every dg-category is h-projective and hence one can compute the derived tensor product by
\[\A \otimes^\mathbf{L} \B = \A^{\op{hp}} \otimes \B = \A \otimes \B.\]

\subsection{DG-Modules}\label{subsection: dg modules}
For any small dg-category, \(\A\), denote by \(\dgMod{\A}\) the dg-category of dg-functors \(\DGCAT{k}\left(\A^{\opp},\CH{k}\right)\), where \(\CH{k}\) denotes the dg-category of chain complexes equipped with the internal Hom from its symmetric monoidal closed structure.
The objects of \(\dgMod{\A}\) will be called {\bf dg \(\A\)-modules}.
Since one may view the dg \(\A^{\opp}\)-modules as what should reasonably be called left dg \(\A\)-modules, the terms right and left will be dropped in favor of dg \(\A\)-modules and dg \(\A^{\opp}\)-modules, respectively.
We note here that the somewhat vexing choice of terminology is such that we can view objects of \(\A\) as dg \(\A\)-modules by way of the enriched Yoneda embedding
\[Y_\A \colon \A \to \dgMod{\A}.\]

As a special case, we define for any two small dg-categories, \(\A\) and \(\B\), the category of dg \(\A\)-\(\B\)-bimodules to be \(\dgMod{\A^{\opp} \otimes \B}\).
We note here that the symmetric monoidal closed structure on \(\dgcat{k}\) allows us to view bimodules as morphisms of dg-categories by the isomorphism
\begin{eqnarray*}
  \dgMod{\A^{\opp} \otimes \B} &=& \DGCAT{k}\left(\A \otimes \B^{\opp}, \CH{k}\right)\\
  &\cong& \DGCAT{k}(\A, \DGCAT{k}\left(\B^{\opp}, \CH{k}\right))\\
  &=& \DGCAT{k}\left(\A, \dgMod{\B}\right).
\end{eqnarray*}
The image of a dg \(\A\)-\(\B\)-bimodule, \(E\), is the dg-functor \(\Phi_E(A) = E(A,-)\).

\subsubsection{h-Projective DG-Modules}
We say that a dg \(\A\)-module, \(N\), is {\bf acyclic} if \(N(A)\) is an acyclic chain complex for all objects \(A\) of \(\A\).
A dg \(\A\)-module \(M\) is said to be {\bf h-projective} if
\[H^0(\dgMod{\A})(M,N) := H^0(\dgMod{A}(M,N)) = 0\]
for every acyclic dg \(\A\)-module, \(N\).
The full dg-subcategory of \(\dgMod{\A}\) consisting of h-projectives will be called \(\hproj{\A}\).

We always have a special class of h-projectives given by the representables, \(h_A = \A(-, A)\) for if \(M\) is acyclic, then from the enriched Yoneda Lemma we have
\[H^0(\dgMod{A})(h_A,M) := H^0(\dgMod{\A}(h_A, M)) \cong H^0(M(A)) = 0.\]
Noting that closure of \(\hproj{\A}\) under homotopy equivalence follows immediately from the Yoneda Lemma applied to \(H^0(\dgMod{A})\), we define \(\overline{\A}\) to be the full dg-subcategory of \(\hproj{\A}\) consisting of the dg \(\A\)-modules homotopy equivalent to representables.

We will say an h-projective dg \(\A\)-\(\B\)-bimodule, \(E\), is {\bf right quasi-representable} if for every object \(A\) of \(\A\) the dg \(\B\)-module \(\Phi_E(A)\) is an object of \(\overline{\B}\), and we will denote by \(\hproj{\A^\opp \otimes \B}^\rqr\) the full subcategory of \(\hproj{\A^\opp \otimes \B}\) consisting of all right quasi-representables.

\subsubsection{The Derived Category of a DG-Category}
By definition, a degree zero closed morphism
\[\eta \in Z^0(\dgMod{\A})(M,N)\]
satisfies
\[\eta(A) \in Z^0(\CH{k}(M(A), N(A))) = \Ch{k}(M(A), N(A))\]
for all objects \(A\) of \(\A\).
Hence we are justified in the following definitions:
\begin{enumerate}[(i)]
\item
  \(\eta\) is a {\bf quasi-isomorphism} if \(\eta(A)\) is a quasi-isomorphism of chain complexes for all objects \(A\) of \(\A\), and
\item
  \(\eta\) is a {\bf fibration} if \(\eta(A)\) is a degree-wise surjective morphism of complexes for all objects \(A\) of \(\A\).
\end{enumerate}
Equipping \(\Ch{k}\) with the standard projective model structure (see \cite[Section 2.3]{Hov98}), these definitions endow \(Z^0(\dgMod{\A})\) with the structure of a particularly nice cofibrantly generated model category (see \cite[Section 3]{Toen}).
In analogy with the definition of the derived category of modules for a ring \(A\), the {\bf derived category of \(\A\)} is defined to be the model category theoretic homotopy category,
\[\mathrm{D}(\A) = \Ho{Z^0(\dgMod{\A})} = Z^0(\dgMod{\A})[\mathcal{W}^{-1}]\]
obtained from localizing \(Z^0(\dgMod{\A})\) at the class, \(\mathcal{W}\), of quasi-isomorphisms.

It can be shown (see \cite[Section 3.5]{Kel95}) that for every dg \(\A\)-module, \(M\), there exists an h-projective, \(N\), and a quasi-isomorphism \(N \to M\), which one calls an {\bf h-projective resolution of \(M\)}.
Moreover, it is not difficult to see that any quasi-isomorphism between h-projective objects is in fact a homotopy equivalence.
It follows that there is an equivalence of categories between \(H^0(\hproj{\A})\) and \(\mathrm{D}(\A)\) for any small dg-category, \(\A\).

It should be noted that this generalizes the notion of derived categories of modules over a commutative ring.
Indeed, for a commutative ring, \(A\), one associates to \(A\) the ringoid, \(\A\), with one object, \(\ast\), and morphisms, \(\A(\ast,\ast)\), the complex with \(A\) in degree zero.
One identifies the chain complexes of \(A\)-modules enriched by the Hom total complex with \(\dgMod{\A}\), which is simply the full dg-subcategory of the functor category \(\op{Fun}(\mathcal{A}, \CH{k})\) comprised of all dg-functors.
From this viewpoint it is easy to recognize the categories \(Z^0(\dgMod{\A})\), \(H^0(\dgMod{\A})\), and \(\mathrm{D}(\A)\), as the categories \(\Ch{A}\), \(K(A)\), the usual category up to homotopy, and the derived category of \(\Mod{A}\), respectively.
In the language of \cite{Lunts-Orlov}, \(\hproj{\A}\) is a dg-enhancement of \(\mathrm{D}(\Mod{A})\).

\subsection{Tensor Products of DG-Modules}

Let \(M\) be a dg \(\A\)-module, let \(N\) be a dg \(\A^{\opp}\)-module, and let \(A, B\) be objects of \(\A\).
For ease of notation, we drop the functor notation \(M(A)\) in favor of \(M_A\) and write \(\A_{A,B}\) for the morphisms \(\A(A, B)\).
We have structure morphisms
\[M_{A,B} \in \CH{k}\left(\A_{A,B}, \CH{k}(M_B, M_A)\right) \cong \CH{k}\left(M_B \otimes_k \A_{A,B}, M_A\right)\]
and
\[N_{A,B} \in \CH{k}\left(\A_{A,B}, \CH{k}(N_A, N_B)\right) \cong \CH{k}\left(\A_{A,B} \otimes_k N_A, N_B\right),\]
which give rise to a unique morphism
\[M_B \otimes_k {A}_{A,B} \otimes_k N_A \ra M_A \otimes_k N_A \oplus M_B \otimes_k N_B\]
induced by the universal properties of the biproduct.
The two collections of morphisms given by projecting onto each factor induced morphisms 
\[\Xi_1, \Xi_2 \colon \bigoplus_{A,B \in \op{Ob}(\A)} M_B \otimes_k \A_{A,B} \otimes_k N_A \ra \bigoplus_{\C \in \op{Ob}(\A)} M_C \otimes_k N_C,\]
and we define the tensor product of \(M\) and \(N\) to be the coequalizer in \(\Ch{k}\)
\[\begin{tikzcd}
\bigoplus_{(i,j) \in \Z^2} M_j \otimes_k \A_{A,B} \otimes_k N_A \arrow[shift left]{r}{\Xi_1} \arrow[shift right,swap]{r}{\Xi_2} & \bigoplus_{\ell \in \Z} M_\ell \otimes_k N_\ell\arrow{r} & M \otimes_\A N
\end{tikzcd}.\]It is routine to check that a morphism \(M \ra M^\prime\) of right dg \(\A\)-modules induces by the universal property for coequalizers a unique morphism
\[M \otimes_\A N \ra M^\prime \otimes_\A N\]
yielding a functor
\[- \otimes_\A N \colon \dgMod{\A} \ra \CH{k}.\]

One extends this construction to bimodules as follows.
Given objects \(E\) of \(\dgMod{\A \otimes \B}\) and \(F\) of \(\dgMod{\B^{\opp} \otimes \C}\), we recall that we have associated to each a dg-functor \(\Phi_E \colon \A^{\opp} \ra \dgMod{\B}\) and \(\Phi_F \colon \C^{\opp} \ra \dgMod{\B^{\opp}}\) by the symmetric monoidal closed structure on \(\dgcat{k}\).
For each pair of objects \(A\) of \(\A\) and \(C\) of \(\C\), we obtain dg-modules
\[\Phi_E(A) = E(A, -) \colon \B^{\opp} \ra \CH{k}\ \text{and}\ \Phi_F(C) = F( -, C) \colon \B \ra \CH{k}\]
and hence one may define the object \(E \otimes_\B F\) of \(\dgMod{\A \otimes \C}\) by
\[\left(E \otimes_\B F\right)(A, C) = \Phi_E(A) \otimes_\B \Phi_F(C).\]
One can show that by a similar argument to the original that a morphism \(E \to E^\prime\) of \(\dgMod{\A \otimes \B}\) induces a morphism \(E \otimes_\B F \to E^\prime \otimes_\B F\) of \(\dgMod{\A \otimes \C}\), and a morphism \(F \to F^\prime\) of \(\dgMod{\B^{\opp} \otimes \C}\) induces a morphism \(E \otimes_\B F \to E \otimes_\B F^\prime\) of \(\dgMod{\A \otimes \C}\).

\begin{remark}\label{rem: tensor over k}
  Denote by \(\K\) the dg-category with one object, \(\ast\), and morphisms given by the chain complex
  \[\K(\ast,\ast)^n =
  \left\{ \begin{matrix}
    k & n = 0\\
    0 & n \neq 0
  \end{matrix}\right.\]  with zero differential.
  This category serves as the unit of the symmetric monoidal structure on \(\dgcat{k}\), so for small dg-categories, \(\A\) and \(\C\), we can always identify \(\A\) with \(\A \otimes \K\) and \(\C\) with \(\K^\opp \otimes \C\).
  With this identification in hand, we obtain from taking \(\B = \K\) in the latter construction a special case:
  Given a dg \(\A^\opp\)-module, \(E\), and a dg \(\C\)-module, \(F\), we have a dg \(\A\)-\(\C\)-bimodule defined by the tensor product
  \[\left(E \otimes F\right)(A,C) := \left(E \otimes_{\K} F\right)(A,C) = E(A) \otimes_k F(C).\]
\end{remark}

\subsection{Extensions of Morphisms Associated to Bimodules}
Let \(E\) be a dg \(\A\)-\(\B\)-bimodule.
Following \cite[Section 3]{CS}, we can extend the associated functor \(\Phi_E\) to a dg-functor
\[\widehat{\Phi_M} \colon \dgMod{\A} \to \dgMod{\B}\]
defined by \(\widehat{\Phi_E}(M) = M \otimes_\A E\).
Similarly, we have a dg-functor in the opposite direction
\[\widetilde{\Phi_M} \colon \dgMod{\B} \to \dgMod{\A}\]
defined by \(\widetilde{\Phi_M}(N) = \dgMod{\B}(\Phi_M( - ), N)\).

For any dg-functor \(G \colon \A \to \B\) we denote by \(\Ind{G}\) the extension of the dg-functor
\[A \to \B \overset{Y_\B}\to \dgMod{\B}\]
and its right adjoint by \(\Res{G}\).
By way of the enriched Yoneda Lemma we see that for any object \(A\) of \(\A\) and any dg \(\B\)-module, \(N\), 
\[\Res{G}(N)(A) = \dgMod{\B}(h_{GA}, N) \cong N(GA).\]

We record here some useful propositions regarding extensions of dg-functors.

\begin{proposition}[{\cite[Prop 3.2]{CS}}]
  Let \(\A\) and \(\B\) be small dg-categories.
  Let \(F \colon \A \to \dgMod{\B}\) and \(G \colon \A \to \B\) be dg-functors.
  \begin{enumerate}[(i)]
  \item
    \(\widehat{F}\) is left adjoint to \(\widetilde{F}\) (hence \(\Ind{G}\) is left adjoint to \(\Res{G}\)),
  \item
    \(\widehat{F} \circ Y_\A\) is dg-isomorphic to \(F\) and \(H^0(\widehat{F})\) is continuous (hence \(\Ind{G} \circ Y_\A\) is dg-isomorphic to \(Y_\B \circ G\) and \(H^0(\Ind{G})\) is continuous),
  \item
    \(\widehat{F}(\hproj{\A}) \subseteq \hproj{\B}\) if and only if \(F(A) \subseteq \hproj{B}\) (hence \(\Ind{G}(\hproj{\A}) \subseteq \hproj{B}\)),
  \item
    \(\Res{G}(\hproj{\B}) \subseteq \hproj{\A}\) if and only if \(\Res{G}(\bar{B}) \subseteq \hproj{\A}\); moreover,\\ \(H^0(\Res{G})\) is always continuous,
  \item
    \(\Ind{G} \colon \hproj{\A} \to \hproj{\B}\) is a quasi-equivalence if \(G\) is a quasi-equivalence.
  \end{enumerate}
\end{proposition}

\begin{remark}\label{rem: tensoring with reps}
  \begin{enumerate}
  \item
    We note that for dg \(\A\)- and \(\A^\opp\)-modules, \(M\) and \(N\), part \((i)\) implies that the dg-functors
    \[- \otimes_\A N \colon \dgMod{\A} \to \CH{k}\ \text{and}\ M \otimes_\A - \colon \dgMod{\A^\opp} \to \CH{k}\]
    have right adjoints
    \[\widetilde{N}(C) = \CH{k}(N( - ), C)\ \text{and}\ \widetilde{M}(C) = \CH{k}(M( - ), C),\]
    respectively.
    As an immediate consequence of the enriched Yoneda Lemma
    \[h_A \otimes_\A N \cong N(A)\ \text{and}\ M \otimes_\A h^A \cong M(A)\]
    holds for any object \(A\) of \(\A\).
  \item
    We denote by \(\Delta_\A\) the dg \(\A\)-\(\A\)-bimodule corresponding to the Yoneda embedding, \(Y_\A\), under the isomorphism
    \[\dgMod{\A^\opp \otimes \A} \cong \DGCAT{k}\left(\A, \dgMod{\A}\right).\]
    It's clear that we have a dg-functor
    \[\Delta_\A \otimes_\A - \colon \dgMod{\A^\opp \otimes \A} \to \dgMod{\A^\opp \otimes \A}\]
    and for any dg \(\A\)-\(\A\)-bimodule, \(E\), we see that
    \[(\Delta_\A \otimes_\A E)(A,A^\prime) = h_A \otimes_\A E(- , A^\prime) \cong E(A,A^\prime)\]
    implies that \(\Delta_\A \otimes_\A E \cong E\).
  \end{enumerate}
\end{remark}

When starting with an h-projective we have a very nice extension of dg-functors:
\begin{proposition}[{\cite[Lemma 3.4]{CS}}]
  For any h-projective dg \(\A\)-\(\B\)-bimodule, \(E\), the associated functor
  \[\Phi_E \colon \A \to \dgMod{\B}\]
  factors through \(\hproj{\B}\).
\end{proposition}

As a direct consequence of the penultimate proposition, this means that we can view the extension of \(\Phi_E\) as a dg-functor
\[\widehat{\Phi_E} = - \otimes_\A E \colon \hproj{A} \to \hproj{B}.\]
Put another way, tensoring with an h-projective \(\A\)-\(\B\)-bimodule preserves h-projectives.

One essential result about \(\dgcat{k}\) comes from T\"oen's result on the existence, and description of, the internal Hom in its homotopy category. 

\begin{theorem}[{\cite[Theorem 1.1]{Toen}}, {\cite[4.1]{CS}}] \label{theorem: Toen}
  Let \(\A\), \(\B\), and \(\C\) be objects of \(\dgcat{k}\).
  There exists a natural bijection
  \[[\A,\C] \overset{1:1}\longleftrightarrow \Iso{H^0(\hproj{\A^{\opp} \otimes \C}^\rqr)}\]
  Moreover, the dg-category \(\RHom{\B,\C} := \hproj{\B^\opp \otimes \C}^\rqr\) yields a natural bijection
  \[[\A \otimes \B, \C] \overset{1:1}\longleftrightarrow [\A, \RHom{\B,\C}]\]
  proving that the symmetric monoidal category \(\Ho{\dgcat{k}}\) is closed.
\end{theorem}

\begin{corollary}[{\cite[7.2]{Toen}},{\cite[Cor. 4.2]{CS}}] \label{corollary: Toen}
  Given two dg categories \(\A\) and \(\B\), \\\(\RHom{\A, \hproj{\B}}\) and \(\hproj{\A^\opp \otimes \B}\) are isomorphic in \(\Ho{\dgcat{k}}\).
  Moreover, there exists a quasi-equivalence
  \[\RHomc{\hproj{\A},\hproj{\B}} \to \RHom{\A, \hproj{B}}.\]
\end{corollary}

To get a sense of the value of this result, let us recall one application from \cite[Section 8.3]{Toen}. Let \(X\) and \(Y\) be quasi-compact and separated schemes over \(\op{Spec} k\). 
Recall the dg-model for \(\mathrm{D}(\op{Qcoh} X)\), \(\mathcal L_{\op{qcoh}}(X)\), is the \(\mathcal C(k)\)-enriched subcategory of fibrant-cofibrant objects in the injective model structure on \(C(\op{Qcoh} X)\).

\begin{theorem}[{\cite[Theorem 8.3]{Toen}}]
  Let \(X\) and \(Y\) be quasi-compact, quasi-separated schemes over \(k\). Then there exists an isomorphism in \(\Ho{\dgcat{k}}\)
  \begin{displaymath}
    \RHomc{\mathcal L_{\op{qcoh}} X,\mathcal L_{\op{qcoh}} Y} \cong \mathcal L_{\op{qcoh}} (X \times_k Y)
  \end{displaymath}
  which takes a complex \(E \in \mathcal L_{\op{qcoh}} (X \times_k Y)\) to the exact functor on the homotopy categories
  \begin{align*}
    \Phi_E : \mathrm{D}(\Qcoh{X}) & \to \mathrm{D}(\Qcoh{Y}) \\
    M & \mapsto \mathbf{R}\pi_{2 \ast} \left( E \overset{\mathbf{L}}{\otimes} \mathbf{L}\pi_1^\ast M \right)
  \end{align*}
\end{theorem}

\begin{proof}
  The first part of the statement is exactly as in \cite{Toen}. The second part is implicit. 
\end{proof}

\section{Background on noncommutative projective schemes} \label{section: background on NCP}

\subsection{Recollections and conditions} \label{subsection: standard results and conditions}

Noncommutative projective schemes were introduced by Artin and Zhang in \cite{AZ}. We recall the definition.

\begin{definition}
  Let \(N\) be a finitely-generated abelian group. We say that a \(k\)-algebra \(A\) is \(N\)-graded if there exists a decomposition as \(k\)-modules 
  \begin{displaymath}
    A = \bigoplus_{n \in N} A_n
  \end{displaymath}
  with \(A_n A_m \subset A_{n+m}\). One says that \(A\) is \textbf{connected graded} if it is \(\Z\)-graded with \(A_0 = k\) and \(A_n = 0\) for \(n < 0\). 
\end{definition}

For algebraic geometers, the most common example is the homogenous coordinate ring of a projective scheme. These are of course commutative. One has a plentitude of noncommutative examples. 

\begin{example} \label{example: ncPs}
  Let us take \(k = \mathbf{C}\) and consider the following quotient of the free algebra 
  \begin{displaymath}
    A_q := \mathbf{C}\langle x_0,\ldots,x_n \rangle/(x_i x_j - q_{ij} x_j x_i)
  \end{displaymath}
  for \(q_{ij} \in \mathbf{C}^\times\). These give noncommutative deformations of \(\mathbf{P}^n\).
\end{example}

\begin{example} \label{example: ncCY}
  Building off of Example~\ref{example: ncPs}, we recall the following class of noncommutative algebras of Kanazawa \cite{Kana}. Pick \(\phi \in \mathbf{C}\) and $q_{ij}$ according to \cite[Theorem 2.1]{Kana} with $\prod_{i=1}^n q_{ij} = 1$ for all $j$. And set 
  \begin{displaymath}
    A^\phi_q := A_q / \left( \sum_{i = 0}^n x_i^{n+1} - \phi(n+1)(x_0\cdots x_n) \right). 
  \end{displaymath}
  This is the noncommutative version of the homogeneous coordinate rings of the Hesse (or Dwork) pencil of Calabi-Yau hypersurfaces in \(\mathbf{P}^n\). 
\end{example}

\begin{definition}
  Let \(M\) be a graded \(A\)-module. We say that \(M\) has \textbf{right limited grading} if there exists a \(d\) such \(M_{d^\prime} = 0\) for \(d^\prime \geq d\). We define \textbf{left limited grading} analogously.
\end{definition}

For a connected graded \(k\)-algebra, \(A\), one has the two-sided ideal
\begin{displaymath}
  A_{\geq m} :=  \bigoplus_{n \geq m} A_n.
\end{displaymath}

\begin{definition}
  Let \(A\) be a finitely generated connected graded algebra. Recall that an element, \(m\), of a module, \(M\), is \textbf{torsion} if there is an \(n\) such that
  \begin{displaymath}
    A_{\geq n} m = 0.
  \end{displaymath}
  We let \(\tau\) denote the functor that takes a module, \(M\), to its torsion submodule.
  The module \(M\) is torsion if \(\tau M = M\).
\end{definition}

\begin{proposition}\label{proposition: f.g. torsion}
  Let \(A\) be a connected graded \(k\)-algebra.
  Denote by \(\Tors{A}\), the full subcategory of \(\Gr{A}\) consisting of all torsion modules.
  If \(A\) is finitely generated in positive degree, then \(\Tors{A}\) is a Serre subcategory.
\end{proposition}

\begin{proof}
  Let \(S = \{x_i\}_{i=1}^r\) be a set of generators for \(A\) as a \(k\)-algebra and let \(d_i = \deg(x_i)\).
  Consider a short exact sequence
  \[0 \to M^\prime \to M \overset{p}\to M^{\prime\prime} \to 0.\]
  It's clear that if \(M\) is an object of \(\Tors{A}\), then so are \(M^\prime\) and \(M^{\prime\prime}\).
  Hence it suffices to show that if \(M^\prime\) and \(M^{\prime\prime}\) are both objects of \(\Tors{A}\), then so is \(M\).
  
  First assume that there exists some \(N\) such that for any \((X_1, X_2, \ldots, X_N) \in \prod_{i = 1}^N S\) we have \((X_1 \cdots X_N)m = 0\).
  Let \(d = \max(\{d_i\}_{i = 1}^r)\) and take any \(a \in A_{\geq dN}\).
  By assumption we can write \(a = \sum_{i=1}^n \alpha_i a_i\) with \(\alpha_i \in k\), each \(a_i\) of the form
  \[a_i = X_{i,1} X_{i,2}\cdots X_{i, s_i},\, X_{i,j} \in S\]
  and, for each \(i\),
  \[dN \leq \sum_{j = 1}^{s_i} \deg(X_{i,j}) = \deg(a) \leq ds_i.\]
  It follows that \(N \leq s_i\) and hence \(am = 0\).
  Thus it suffices to find such an \(N\).

  Fix an element \(m \in M\).
  Since \(M^{\prime\prime}\) is an object of \(\Tors{A}\), there exists some \(n\) such that \(A_{\geq n} p(m) = 0\) and hence \(A_{\geq n}m \in M^\prime\).
  In particular, if we let \(T = \prod_{i = 1}^n S\), then for any element \(t = (X_1, X_2, \ldots, X_n) \in T\) we have an element \(a_t = X_1 X_2 \cdots X_n \in A_{\geq n}\) and so \(a_t m \in M^\prime\).
  Let \(n_t\) be such that \(A_{\geq n_t} (a_t m) = 0\) and take \(N = 2\max(\{n_t\}_{t \in T} \cup \{n\}) + 1\).
  If we take any element \((X_1, X_2, \ldots, X_N) \in \prod_{i = 1}^N S\), then we can form an element \(a_t = X_{N - n} X_{N - n + 1} \cdots X_N \in A_{\geq n}\).
  By construction, \(a_t m \in M^\prime\) and \(a_t^\prime = X_1 X_2 \cdots X_{N - n - 1} \in A_{\geq n_t}\) since \(n_t \leq N - n - 1\).
  Therefore we have
  \[0 = a_t^\prime (a_t m) = (X_1 X_2 \cdots X_N) m,\]
  as desired.
\end{proof}

As such, we can form the quotient.

\begin{definition}
  Let \(A\) be connected graded and finitely generated as a \(k\)-algebra. Then denote the quotient of the category of graded \(A\)-modules by torsion as
  \begin{displaymath}
    \QGr{A} := \Gr{A} / \Tors{A}
  \end{displaymath}
  Let
  \begin{displaymath}
    \pi : \Gr{A} \to \QGr{A}
  \end{displaymath}
  denote the quotient functor. By Proposition~\ref{proposition: f.g. torsion} and \cite[Cor. 1, III.3]{DCA62}, \(\pi\) admits a fully faithful right adjoint which we denote by 
  \begin{displaymath}
    \omega : \QGr{A} \to \Gr{A}.
  \end{displaymath}
  Finally, we denote the composition \(Q := \omega \pi\). 
  
  The category \(\QGr{A}\) is defined to be the quasi-coherent sheaves on the \textbf{noncommutative projective scheme} \(X\). 
\end{definition}

\begin{remark}
  Note that, traditionally speaking, \(X\) is not a space, in general. In the case \(A\) is commutative and finitely-generated by elements of degree \(1\), then a famous result of Serre says that \(X\) is \(\op{Proj} A\). 
\end{remark}

One can give more explicit descriptions of \(Q\) and \(\tau\). 

\begin{proposition}
  Let \(A\) be a finitely generated connected graded \(k\)-algebra and let \(M\) be a graded \(A\)-module. Then 
  \begin{align*}
    \tau M & = \op{colim}_n \GR{A}(A/A_{\geq n}, M) \\
    Q M & = \op{colim}_n \GR{A}(A_{\geq n}, M).
  \end{align*}
\end{proposition}

\begin{proof}
  This is standard localization theory, see \cite{Stenstrom}.
\end{proof}

In studying questions of kernels and bimodules, we will have to move outside the realm of \(\Z\)-gradings.
While one can generally treat \(N\)-graded \(k\)-algebras in our analysis, we limit the scope a bit and only consider \(\Z^2\)-gradings of the following form.

\begin{definition}
  Let \(A\) and \(B\) be connected graded \(k\)-algebras. The tensor product \(A \otimes_k B\) will be equipped with its natural bi-grading 
  \begin{displaymath}
    (A \otimes_k B)_{n_1,n_2} = A_{n_1} \otimes_k B_{n_2}. 
  \end{displaymath}
  A \textbf{bi-bi module} for the pair \((A,B)\) is a \(\Z^2\)-graded \(A \otimes_k B\) module. 
\end{definition}

\begin{remark}
  As noted in the remarks above \cite[Lemma 4.1]{VdB}, the notion of \(A\)-torsion and \(B\)-torsion bi-bi modules is well-defined provided that \(A\) and \(B\) are finitely generated as \(k\)-algebras.
  From this point on, all of our \(k\)-algebras will be assumed to be finitely generated.
\end{remark}

There are a couple notions of torsion for a bi-bi module that one can dream up. We use the following.

\begin{definition}
  Let \(A\) and \(B\) be finitely generated, connected graded \(k\)-algebras, and let \(M\) be a bi-bi \(A\)-\(B\) module. We say that \(M\) is \textbf{torsion} if it lies in the smallest Serre subcategory containing \(A\)-torsion bi-bi modules and \(B\)-torsion bi-bi modules.
\end{definition}

In the case that \(A=B\), there is a particular bi-bi module of interest.

\begin{definition}
  For \(A\) a finitely generated, connected graded \(k\)-algebra, we define \(\Delta_A\) to be the \(A\)-\(A\) bi-bi module with 
  \begin{displaymath}
    (\Delta_A)_{i,j} = A_{i+j}
  \end{displaymath}
  and the natural left and right \(A\) actions. If the context is clear, we will often simply write \(\Delta\). 
\end{definition}

\begin{lemma} \label{lemma: alternate char of bibi torsion}
  Let \(A\) and \(B\) be finitely generated, connected graded \(k\)-algebras.
  A bi-bi module \(M\) is torsion if and only if there exists \(n_1,n_2\) such that 
  \begin{displaymath}
    (A \otimes B)_{\geq n_1, \geq n_2} m = 0
  \end{displaymath}
  for all \(m \in M\).
\end{lemma}

\begin{proof}
  Note that if \(M\) is \(A\)-torsion then \((A \otimes B)_{\geq n, \geq 0} m = 0\) for some \(n\) for each \(m \in M\). Similarly if \(M\) is \(B\)-torsion then \((A \otimes B)_{\geq 0,\geq n} M = 0\) for some \(n\). So it suffices to show that if
  \begin{displaymath}
    (A \otimes B)_{\geq n_1, \geq n_2} m = 0 , \forall m \in M
  \end{displaymath}
  then it lies in the Serre category generated by \(A\) and \(B\) torsion. Let \(\tau_B M\) be the \(B\)-torsion submodule of \(M\) and consider the quotient \(M/\tau_B M\). For \(m \in M\), we have \(A_{\geq n_1}m\) is \(B\)-torsion, so its image in the quotient \(M/\tau_B M\) is \(A\)-torsion. Consequently, \(M / \tau_B M\) is \(A\)-torsion itself and \(M\) is an extension of \(B\)-torsion and \(A\)-torsion. 
\end{proof}

One can form the quotient category
\begin{displaymath}
  \QGr{A \otimes_k B} := \Gr{A \otimes_k B} / \Tors{A \otimes_k B}.
\end{displaymath}

\begin{lemma} \label{lemma: biQ and bQGr}
  The quotient functor 
  \begin{displaymath}
    \pi : \Gr{A \otimes_k B} \to \QGr{A \otimes_k B}
  \end{displaymath}
  has a fully faithful right adjoint 
  \begin{displaymath}
    \omega : \QGr{A \otimes_k B} \to \Gr{A \otimes_k B}
  \end{displaymath}
  with 
  \begin{displaymath}
    QM := \omega \pi M = \op{colim}_{n_1,n_2} \GR{(A \otimes_k B)}( A_{\geq n_1} \otimes_k B_{\geq n_2} , M)
  \end{displaymath}
\end{lemma}

\begin{proof}
  This is just an application of \cite[Cor. 1, III.3]{DCA62}.
\end{proof}

\begin{corollary}\label{corollary: relation on Qs}
  We have an isomorphism 
  \begin{displaymath}
    Q_{A \otimes_k B} \cong Q_A \circ Q_B \cong Q_B \circ Q_A
  \end{displaymath}
\end{corollary}

\begin{proof}
  This follows from Lemma~\ref{lemma: biQ and bQGr} using tensor-Hom adjunction. 
\end{proof}

We also have the following standard triangles of derived functors. 

\begin{lemma} \label{lemma: exact triangles}
  Let \(A\) and \(B\) be finitely generated connected graded algebras. Then, we have natural transformations 
  \begin{displaymath}
    \mathbf{R} \tau \to \op{Id} \to \mathbf{R} Q 
  \end{displaymath}
  which when applied to any graded module \(M\) gives an exact triangle 
  \begin{displaymath}
    \mathbf{R} \tau M \to M \to \mathbf{R} Q M.
  \end{displaymath}
\end{lemma}

\begin{proof}
  Before we begin the proof, we clarify the statement. The conclusions hold for graded \(A\) (or \(B\)) modules and for bi-bi modules. Due to the formal properties, it is economical to keep the wording of the theorem as so since any reasonable interpretation yields a true statement. 
  
  For the case of graded \(A\) modules, this is well-known, see \cite[Property 4.6]{BV}. For the case of bi-bi \(A \otimes_k B\) modules, the natural transformations are obvious. For each \(M\), the sequence 
  \begin{displaymath}
    0 \to \tau M \to M \to Q M
  \end{displaymath}
  is exact. It suffices to prove that if \(M = I\) is injective, then the whole sequence is actually exact. Here one can use the system of exact sequences
  \begin{displaymath}
    0 \to A_{\geq n_1} \otimes_k B_{\geq n_2} \to A \otimes_k B \to (A \otimes_k B) / A_{\geq n_1} \otimes_k B_{\geq n_2} \to 0
  \end{displaymath}
  and exactness of \(\op{Hom}(-,I)\) plus Lemma~\ref{lemma: alternate char of bibi torsion} to get exactness. 
\end{proof}

In general, good behavior of \(\QGr{A}\) occurs with some homological assumptions on the ring \(A\). We recall two common such ones. 

\begin{definition} \label{definition: Ext-finite}
  Let \(A\) be a connected graded \(k\)-algebra. Following van den Bergh \cite{VdB}, we say that \(A\) is \textbf{Ext-finite} if for each \(n \geq 0\) the ungraded Ext-groups are finite dimensional 
  \begin{displaymath}
    \op{dim}_k \op{Ext}_A^n (k,k) < \infty.
  \end{displaymath}
\end{definition}

\begin{remark}
  The Ext's are taken in the category of left \(A\)-modules, a priori.
  Moreover, as noted in the opening remarks of \cite[Section 4.1]{BV}, if \(A\) is Ext-finite, then \(A\) is finitely presented.
\end{remark}

\begin{definition} \label{definition: chi}
  Following Artin and Zhang \cite{AZ}, given a graded left module \(M\), we say \(A\) satisfies \(\chi^\circ(M)\) if \(\underline{\op{Ext}}^n_A(k,M)\) has right limited grading for each \(n \geq 0\). 

\end{definition}

\begin{remark}
  The equivalence of these two definitions is \cite[Proposition 3.8 (1)]{AZ}.
\end{remark}

We recall some basic results on Ext-finiteness, essentially from \cite[Section 4]{VdB}.

\begin{proposition} \label{proposition: tensor and op properties of ext-finite}
  Assume that \(A\) and \(B\) are Ext-finite. Then
  \begin{enumerate}
  \item the ring \(A \otimes_k B\) is Ext-finite. 
  \item the ring \(A^\opp\) is Ext-finite.
  \end{enumerate}
\end{proposition}

\begin{proof}
  See \cite[Lemma 4.2]{VdB} and the discussion preceeding it. 
\end{proof}

\begin{proposition} \label{proposition: derived Q commutes with coproducts}
  Assume that \(A\) is Ext-finite. Then \(\mathbf{R}\tau_A\) and \(\mathbf{R}Q_A\) both commute with coproducts. 
\end{proposition}

\begin{proof}
  See \cite[Lemma 4.3]{VdB} for \(\mathbf{R}\tau_A\). Since coproducts are exact, using the triangle
  \begin{displaymath}
    \mathbf{R}\tau_A M \to M \to \mathbf{R}Q_A M 
  \end{displaymath}
  we see that \(\mathbf{R}\tau_A\) commutes with coproducts if and only if \(\mathbf{R}Q_A\) commutes with coproducts. 
\end{proof}

\begin{corollary} \label{corollary: Q preserves bimodules}
  Let \(A\) and \(B\) be finitely generated, connected graded \(k\)-algebras, and let \(P\) be a chain complex of bi-bi \(A \otimes_k B\) modules. Assume \(\mathbf{R}Q_A\) commutes with coproducts. Then, \(\mathbf{R}Q_A P\) is naturally also a chain complex of bi-bi modules. In particular, if \(A\) is Ext-finite, \(\mathbf{R}Q_A P\) has a natural bi-bi structure. 
\end{corollary}

\begin{proof}
  Note we already have an \(A\)-module structure so we only need to provide a \(\Z^2\) grading and a \(B\)-action. If we write 
  \begin{displaymath}
    P = \bigoplus_{v \in \Z} P_{\ast,v}
  \end{displaymath}
  as a direct sum of left graded \(A\)-modules, then we set 
  \begin{displaymath}
    (\mathbf{R}Q_A P)_{u,v} : = ( \mathbf{R} Q_A (P_{\ast,v}) )_u.
  \end{displaymath}
  The \(B\) module structure is precomposition with the \(B\)-action on \(P\). The only non-obvious condition of the bi-bi structure is that 
  \begin{displaymath}
    \mathbf{R}Q_A P = \bigoplus_{u,v} (\mathbf{R}Q_A P)_{u,v}
  \end{displaymath}
  which is equivalent to pulling the coproduct outside of \(\mathbf{R}Q_A\). We can do this for Ext-finite \(A\) thanks to Proposition~\ref{proposition: derived Q commutes with coproducts}. 
\end{proof}

\begin{corollary} \label{corollary: natural maps between Qs}
  Assume that \(A\) and \(B\) are left Noetherian, and that \(\mathbf{R}\tau_A\) and \(\mathbf{R}\tau_B\) both commute with coproducts. There exist natural morphisms of bimodules 
  \begin{align*}
    \beta^l_P & : \mathbf{R}Q_{A} P \to \mathbf{R}Q_{A \otimes_k B} P \\
    \beta^r_P & : \mathbf{R}Q_{B} P \to \mathbf{R}Q_{A \otimes_k B} P.
  \end{align*}
\end{corollary}

We first record a more explicit version of \cite[Prop 7.1 (5)]{AZ}, which states that every injective object of \(\Gr{A}\) is of the form \(I_1 \oplus I_2\), with \(I_1\) a torsion-free injective and \(I_2\) an injective torsion module.

\begin{lemma} \label{lemma: Gr injectives}
  Every injective \(I\) of \(\Gr{A}\) is isomorphic to \(\tau_A I \oplus Q_A I\).

\end{lemma}

\begin{proof}
  By \cite[Prop 2.2]{AZ} any essential extension of a torsion module is torsion, so if \(\tau_{A} I \to E\) is an injective envelope, then we have a monic extension over the inclusion of the torsion submodule
  \[\begin{tikzcd}
  0 \arrow{r} & \tau_{A} I \arrow{r}\arrow{d} & E \arrow[dashed]{ld}{\exists}\\
  & I
  \end{tikzcd}\]  since the injective envelope is an essential monomorphism.
  By maximality of \(\tau_{A} I\) amongst torsion subobjects of \(I\), it follows that \(\tau_{A} I = E\) is injective.
  Denoting by \(\varepsilon_A\) the unit of the adjunction
  \(\begin{tikzcd}
  \pi_A \dashv \omega_A \colon \Gr{A} \arrow[shift left=.5ex]{r} & \QGr{A} \arrow[shift left=.5ex]{l}
  \end{tikzcd}\)
  the exact sequence
  \[\begin{tikzcd}
  0 \arrow{r} & \tau_{A} I \arrow{r} & I \arrow{r}{\varepsilon_A(I)} & Q_{A} I \arrow{r} & 0
  \end{tikzcd}\]  splits, as desired.
\end{proof}

\begin{proof}[Proof of {\ref{corollary: natural maps between Qs}}]
  Thanks to Corollary~\ref{corollary: Q preserves bimodules}, we see that the question is well-posed. We handle the case of \(\beta^l_P\) and note that case of \(\beta^r_P\) is the same argument, mutatis mutandis.

  First we make some observations about objects of \(\Gr{\left(A \otimes_k B\right)}\).
  If we regard such an object, \(E\), as an \(A\)-module, the \(A\)-action is
  \[a \cdot e = (a \otimes 1) \cdot e\]
  and we can view \(\tau_A E\) as the elements \(e\) of \(E\) for which
  \[a \cdot e = (a \otimes 1) \cdot e = 0\]
  whenever \(a \in A_{\geq m}\) for some \(m \in \Z\).
  As such, \(\tau_A E\) inherits a bimodule structure from \(E\) and \(\Z^2\)-grading \((\tau_A E)_{u,v} = (\tau_A E_{*,v})_u\) coming from the decomposition
  \[\tau_A E = \tau_A \bigoplus_v E_{\ast,v} \cong \bigoplus_v \tau_A E_{\ast,v}.\]
  Thanks to Lemma~\ref{lemma: alternate char of bibi torsion}, we can view \(\tau_{A \otimes_k B} E\) as the elements \(e\) of \(E\) for which there exists integers \(m\) and \(n\) such that \(a \otimes b \cdot e = 0\) for all \(a \in A_{\geq m}\) and \(b \in B_{\geq n}\).
  From this viewpoint it's clear that
  \[a \otimes b \cdot e = (1 \otimes b) \cdot (a \otimes 1 \cdot e)\]
  implies \(\tau_A E\) includes into \(\tau_{A \otimes_k B} E\).

  We equip \(\Ch{\Gr{A}}\) with the injective model structure and use the methods of model categories to compute the derived functors (see \cite{Hov01} for more details).
  Since we can always replace \(P\) by a quasi-isomorphic fibrant object, we can assume that each \(P^n\) is an injective graded \(A \otimes_k B\)-module.
  Moreover, the fact that the canonical morphisms \(A \to A \otimes_k B\) is flat implies that the associated adjunction is Quillen, and hence \(P\) is fibrant when regarded as an object of  \(\Ch{\Gr{A}}\).
  Since \(Q_A\) preserves injectives, it follows that each \(Q_A P^n\) is an injective object of \(\Gr{A}\).
  It's clear from the fact that \(\tau_A P^n\) is an \(A \otimes_k B\)-module that   \[0 \to \tau_A P^n \to P^n \to P^n/\tau_A P^n \to 0\]
  is an exact sequence of \(\Gr{(A \otimes_k B)}\) for each \(n\).
  Moreover, by Lemma~\ref{lemma: Gr injectives} we have \(P^n/\tau_A P^n \cong Q_A P^n\).
  We thus define \((\beta^l_P)^n\) to be the epimorphism induced by the unversal property for cokerenels as in the commutative diagram
  \[\begin{tikzcd}
  0 \arrow{r} & \tau_{A} P^n \arrow{d}\arrow{r} & P^n \arrow{d}{\id_{P^n}}\arrow{rr}{\varepsilon_{A}(P^n)} && Q_{A} P^n \arrow{r}\arrow[dashed]{d}{\exists ! (\beta^l_P)^n} & 0\\
  0 \arrow{r} & \tau_{A \otimes_k B} P^n \arrow{r} & P^n\arrow{rr}{\varepsilon_{A \otimes_k B}(P^n)} && Q_{A \otimes_k B} P^n \arrow{r} & 0 
  \end{tikzcd}\]  We observe here that by the Snake Lemma, \((\beta^l_P)^n\) is an isomorphism if and only if \(\tau_{A \otimes_k B} P^n \cong \tau_A P^n\), which is equivalent by the remarks above to the condition that \(\tau_B \tau_A P^n = \tau_A P^n\).

  To see that \(\beta\) actually defines a morphism of complexes, we have by naturality of \(\varepsilon_A\), \(\varepsilon_{A \otimes_k B}\), and the commutative diagram defining \((\beta^l_P)^n\) above 
  \begin{eqnarray*}
    (\beta^l_P)^{n+1} \circ Q_{A}(d^n_{P}) \circ \varepsilon_{A}(P^n)
    &=& (\beta^l_P)^{n+1} \circ \varepsilon_{A}(P^{n+1}) \circ d^n_{P}\\
    &=& \varepsilon_{A \otimes_k B}(P^{n+1}) \circ d^n_{P}\\
    &=& Q_{A \otimes_k B}(d_{P}^n) \circ \varepsilon_{A \otimes_k B}(P^n)\\
    &=& Q_{A \otimes_k B}(d_{P}^n) \circ (\beta^l_P)^n \circ \varepsilon_{A}(P^n)
  \end{eqnarray*}
  implies
  \[(\beta^l_P)^{n+1} \circ Q_{A}(d^n_{P}) = Q_{A \otimes_k B}(d^n_{P}) \circ (\beta^l_P)^n\]
  because \(\varepsilon_{A \otimes_k B}(P^n)\) is epic. Hence we have a morphism
  \[\beta^l_P \colon \mathbf{R}Q_{A} P = Q_{A} P \to Q_{A \otimes_k B} P = \mathbf{R}Q_{A \otimes_k B} P.\]

  For naturality, we note that as the fibrant replacement is functorial if we have a morphism of bi-bi modules, then there is an induced morphism of complexes \(\varphi \colon P_1 \to P_2\) between the replacements and for each \(n\) a commutative diagram 
  \[\begin{tikzcd}[column sep=large]
  P_1^n \arrow{d}{\varphi^n}\arrow{r}{\varepsilon_{A}(P_1^n)} & Q_{A} P_1^n\arrow{d}{Q_{A}(\varphi^n)} \arrow{r}{(\beta^l_{P_1})^n} & Q_{A \otimes_k B}P_1^n \arrow{d}{Q_{A \otimes_k B}(\varphi^n)}\\
  P_2^n \arrow{r}{\varepsilon_{A}(P_2^n)} & Q_{A} P_2^n \arrow{r}{(\beta^l_{P_2})^n} & Q_{A \otimes_k B} P_2^n
  \end{tikzcd}\]  The left square commutes by naturality of \(\varepsilon_{A}\) and the right square commutes because
  \begin{eqnarray*}
    (\beta^l_{P_2})^n \circ Q_{A}(\varphi^n) \circ \varepsilon_{A}(P_1^n)
    &=& (\beta^l_{P_2})^n \circ \varepsilon_{A}(P_2^n) \circ \varphi^n\\
    &=&  \varepsilon_{A \otimes_k B}(P_2^n) \circ \varphi^n\\
    &=& Q_{A \otimes_k B}(\varphi^n) \circ \varepsilon_{A \otimes_k B}(P_1^n)\\
    &=& Q_{A \otimes_k B}(\varphi^n) \circ (\beta^l_{P_1})^n \circ \varepsilon_{A}(P_1^n)
  \end{eqnarray*}
  and \(\varepsilon_{A}(P_1^n)\) is epic.
\end{proof}

\begin{proposition} \label{proposition: bi-torsion is a composition}
  Assume that \(A\) and \(B\) are left Noetherian and Ext-finite. Then, we have natural quasi-isomorphisms 
  \begin{align*}
    \mathbf{R}Q_B(\beta^l_P) & : \mathbf{R}Q_B(\mathbf{R}Q_A P) \to \mathbf{R}Q_{A \otimes_k B} P \\
    \mathbf{R}Q_A(\beta^r_P) & : \mathbf{R}Q_A(\mathbf{R}Q_B P) \to \mathbf{R}Q_{A \otimes_k B} P.
  \end{align*}
  Consequently, \(\beta^l_P\) (respectively \(\beta^r_P\)) is an isomorphism if and only if \(\mathbf{R}Q_A P\) (respectively \(\mathbf{R}Q_B P\)) is \(Q_B\) (respectively \(Q_A\)) torsion-free.
\end{proposition}

\begin{proof}
  As above, we can replace \(P\) with a quasi-isomorphic fibrant object, so it suffices to assume that \(P\) is fibrant.
  We see from Corollary~\ref{corollary: relation on Qs} that
  \[\mathbf{R}Q_{A \otimes_k B} P \cong Q_{A \otimes_k B} P \cong Q_B \circ Q_A P \cong \mathbf{R}(Q_B \circ Q_A) P\]
  The result now follows from the natural isomorphism (see, e.g., \cite[Theorem 1.3.7]{Hov98})
  \[\mathbf{R}Q_B \circ \mathbf{R}Q_A \to \mathbf{R}(Q_B \circ Q_A)\]
\end{proof}

Using the standard homological assumptions above, one has better statements for \(P = \Delta\). 

\begin{proposition} \label{proposition: when beta is an isomorphism}
  Let \(A\) be left (respectively, right) Noetherian and assume that the condition \(\chi^\circ(A)\) holds (respectively, as an \(A^\opp\)-module).
  Then the morphism \(\beta^l_{\Delta}\) (respectively, \(\beta^r_{\Delta}\)) of Corollary~\ref{corollary: natural maps between Qs} is a quasi-isomorphism.
  
\end{proposition}

\begin{proof}
  We have a triangle in \(\mathrm{D}(\Gr{A \otimes_k A^\opp})\)
  \[\mathbf{R}\tau_{A^\opp} (\mathbf{R}Q_A \Delta) \to \mathbf{R}Q_A \Delta \to \mathbf{R}Q_{A^\opp} (\mathbf{R}Q_A \Delta) \to \mathbf{R}\tau_{A^\opp} (\mathbf{R}Q_A \Delta)[1].\]
  By Proposition~\ref{proposition: bi-torsion is a composition}, \(\mathbf{R}Q_{A^\opp}(\mathbf{R}Q_A \Delta) \cong \mathbf{R}Q_{A \otimes_k A^\opp} \Delta\), so it suffices to show that we have \(\mathbf{R}\tau_{A^\opp}(\mathbf{R}Q_A \Delta) = 0\).
  Applying \(\mathbf{R}\tau_{A^\opp}\) to the triangle
  \[\mathbf{R}\tau_A \Delta \to \Delta \to \mathbf{R}Q_A \Delta \to \mathbf{R}\tau_A \Delta[1]\]
  we obtain the triangle
  \[\mathbf{R}\tau_{A^\opp} (\mathbf{R}\tau_A \Delta) \to \mathbf{R}\tau_A \Delta \to \mathbf{R}\tau_{A^\opp} (\mathbf{R}Q_A \Delta) \to \mathbf{R}\tau_{A^\opp} (\mathbf{R}\tau_A \Delta)[1]\]
  and so we are reduced to showing that
  \[\mathbf{R}\tau_A \Delta \cong \mathbf{R}\tau_{A \otimes_k A^\opp}\Delta \cong \mathbf{R}\tau_{A^\opp} (\mathbf{R}\tau_A \Delta)\]
  which then implies that \(\mathbf{R}\tau_A^\opp(\mathbf{R}Q_A \Delta) = 0\), as desired.

  First we note that for any bi-bi module, \(P\), the natural morphism
  \[\mathbf{R}\tau_{A^\opp} P \to P\]
  is a quasi-isomorphism if and only if the natural morphism
  \[\mathbf{R}\tau_{A^\opp} P_{x,\ast} \to P_{x,\ast}\]
  is a quasi-isomorphism.
  Moreover, for a right \(A\)-module, \(M\), if \(H^j(M)\) is right limited for each \(j\) then \(\mathbf{R} \tau_{A^\opp} M \to M\) is a quasi-isomorphism,
  so it suffices to show that \((\mathbf{R}^j \tau_A \Delta)_{x,\ast}\) has right limited grading for each \(x\) and \(j\).
  Now, by \cite[Cor. 3.6 (3)]{AZ}, for each \(j\)
  \[\mathbf{R}^j\tau_A(\Delta)_{x,y} = \mathbf{R}^j\tau_A(\Delta_{\ast,y})_x = \mathbf{R}^j\tau_A(A(y))_x = 0\]
  for fixed \(x\) and sufficiently large \(y\).
  This implies that the natural morphism
  \[\mathbf{R}\tau_{A^\opp}(\mathbf{R}\tau_A(\Delta)_{x,\ast}) \to \mathbf{R}\tau_A\Delta_{x,\ast}\]
  is a quasi-isomorphism, as desired.
\end{proof}

Similar hypotheses of Proposition~\ref{proposition: when beta is an isomorphism} will appear often so we attach a name. 

\begin{definition} \label{definition: delightful couple}
  Let \(A\) and \(B\) be connected graded \(k\)-algebras. If \(A\) is Ext-finite, left and right Noetherian, and satisfies \(\chi^\circ(A)\) and \(\chi^\circ(A^\opp)\) then we say that \(A\) is \textbf{delightful}. If \(A\) and \(B\) are both delightful, then we say that \(A\) and \(B\) form a \textbf{delightful couple}. 
\end{definition}

\subsection{Segre Products}
\begin{definition}\label{def: segre product}
  Let \(A\) and \(B\) be connected graded \(k\)-algebras.
  The \textbf{Segre product} of \(A\) and \(B\) is the graded \(k\)-algebra
  \[ A \times_k B = \bigoplus_{0 \leq i} A_i \otimes_k B_i.\]
\end{definition}

\begin{proposition}\label{proposition: segre product of generated in degree 1 is generated in degree 1}
  If \(A\) and \(B\) are connected graded \(k\)-algebras that are finitely generated in degree one, then \(A \times_k B\) is finitely generated in degree one.
\end{proposition}

\begin{proof}
  Let \(S = \{x_i\}_{i = 1}^r \subseteq A_1\) and \(T = \{y_i\}_{i = 1}^s \subseteq B_1\) be generators.
  Take a homogenous element \(a \otimes b \in A_d \otimes_k B_d\).
  We can write
  \[a = \sum_{i = 1}^m \alpha_i X_1^{(i)} \cdots X_d^{(i)}\ \text{and}\ b = \sum_{j = 1}^n \beta_j Y_1^{(j)} \cdots Y_d^{(j)}\]
  for \(\alpha_i,\beta_j \in k\), \((X_1^{(i)}, \ldots, X_d^{(i)}) \in \prod_{i = 1}^d S\), and \((Y_1^{(j)}, \ldots, Y_d^{(j)}) \in \prod_{i = 1}^d T\).
  Hence we have
  \begin{eqnarray*}
    a \otimes b &=& \left(\sum_{i=1}^m \alpha_i X_1^{(i)} \cdots X_d^{(i)}\right) \otimes \left(\sum_{j = 1}^n \beta_j Y_1^{(j)} \cdots Y_d^{(j)}\right)\\
    &=& \sum_{i = 1}^m\left(\alpha_i X_1^{(i)} \cdots X_d^{(i)} \otimes \left(\sum_{j = 1}^n \beta_j Y_1^{(j)} \cdots Y_d^{(j)}\right)\right)\\
    &=& \sum_{i = 1}^m\left(\sum_{j = 1}^n\left(\alpha_i  X_1^{(i)} \cdots X_d^{(i)} \otimes \beta_j Y_1^{(j)} \cdots Y_d^{(j)}\right)\right)\\
    &=& \sum_{i,j} \alpha_i\beta_j (X_1^{(i)} \otimes Y_1^{(j)}) \cdots (X_d^{(i)} \otimes Y_d^{(j)})
  \end{eqnarray*}
    Therefore \(A \times_k B\) is finitely generated in degree one by \(\{x_i \otimes y_j\}_{i,j}\).
\end{proof}

As a nice corollary, we can relax the conditions on \cite[Theorem 2.4]{VR96} to avoid the Noetherian conditions on the Segre and tensor products.

\begin{theorem}[{\cite[Theorem 2.4]{VR96}}]\label{theorem: Van Rompay}
    Let \(A\) and \(B\) be finitely generated, connected graded \(k\)-algebras, and let \(S = A \times_k B\), \(T = A \otimes_k B\).
    If \(A\) and \(B\) are both generated in degree one, then there is an equivalence of categories
    \[\begin{tikzcd}[row sep=tiny]
    \mathbb{V} \colon \QGr{S} \arrow{r} & \QGr{T}\\
    E \arrow[mapsto]{r} & \pi_{T}\left(T \otimes_S \omega_{S}E\right)
    \end{tikzcd}\]
\end{theorem}

\begin{proof}
  As noted in Van Rompay's comments preceding the Theorem, the hypothesis is necessary only to ensure that \(\QGr{S}\) and \(\QGr{T}\) are well-defined.
  Thanks to Proposition~\ref{proposition: f.g. torsion} and Lemma~\ref{lemma: alternate char of bibi torsion}, the equivalence follows by running the same argument.
\end{proof}

\subsection{A Comparison with the Commutative Situation}

To provide a touchstone for the reader, we interpret the definitions and results when \(A\) and \(B\) are commutative and finitely-generated by elements of weight \(1\). Then, \(A = k[x_1,\ldots,x_n]/I_A\) and \(B = k[y_1,\ldots,y_m]/I_B\) for some homogenous ideals \(I_A,I_B\). So \(\op{Spec} A\) is a closed \(\mathbf{G}_m\)-stable subscheme of affine space \(\mathbf{A}^n\) and similarly for \(\op{Spec} B\). Let \(X\) and \(Y\) be the associated projective schemes. Then, 
\begin{displaymath}
  \op{Spec} A \otimes_k B \subset \mathbf{A}^{n+m}
\end{displaymath}
is \(\mathbf{G}_m^2\)-stable. The category \(\Gr{A \otimes_k B}\) is equivalent to the \(\mathbf{G}_m^2\)-equivariant quasi-coherent sheaves on \(\op{Spec} A \otimes_k B\) with \(\Tors{A \otimes_k B}\) being those modules supported on the subscheme corresponding to \((x_1,\ldots,x_n)(y_1,\ldots,y_m)\). Descent then gives that 
\begin{displaymath}
  \QGr{\left(A \otimes_k B\right)} \cong \op{Qcoh}(X \times Y). 
\end{displaymath}

\subsection{Graded Morita Theory} \label{subsection: graded Morita}

This section demonstrates how the tools of dg-categories yield a nice perspective on derived graded Morita. Compare with the well-known graded Morita statement in \cite{Zhang}. 

In order to utilize the machinery of dg-categories, we must first translate chain complexes of graded modules into dg-categories.
While one can na\"ively regard this category as a dg-category by way of an enriched Hom entirely analogous to the ungraded situation, the relevant statements of \cite{Toen} are better suited to the perspective of functor categories.
As such, we adapt the association of a ringoid with one object to a ring from Section~\ref{subsection: dg modules} to the graded situation, considering instead a ringoid with multiple objects.

Throughout this section, we will let \(G = (G, +)\) be an abelian group, and let \(A\) and \(B\) be not necessarily commutative \(G\)-graded algebras over \(k\).
We will generally be concerned with the groups \(\Z\) and \(\Z^2\).
In the sequel, there will be many instances where there are two simultaneous gradings on an object: homological degree and homogenous degree. 
We avoid the latter term, preferring weight, and use degree solely when referring to homological degree.

For clarity, consider the example of a complex of \(G\)-graded left \(A\)-modules, \(M\).
The degree \(n\) piece of \(M\) is the \(G\)-graded left \(A\)-module \(M^n\).
The weight \(g\) piece of the graded module \(M^n\) is the \(A_0\)-module of homogenous elements of (graded) degree \(g\), \(M^n_g\).
Note that in this terminology, the usual morphisms of graded modules are the weight zero morphisms.

\begin{definition}
  Denote by \(\ldgGrMod{A}\) the dg-category with objects chain complexes of \(G\)-graded left \(A\)-modules and morphisms defined as follows.
  
  We say that a morphism \(f \colon M \to N\) of degree \(p\) is a collection of morphisms \(f^n \colon M^n \to N^{n+p}\) of weight zero.
  We denote by \(\ldgGrMod{A}\left(M,N\right)^p\) the collection of all such morphisms, which we equip with the differential
  \[d(f) = d_N \circ f + (-1)^{p+1}f \circ d_M\]
  and define \(\ldgGrMod{A}(M,N)\) to be the resulting chain complex.
  Composition is the usual composition of graded morphisms.

  We denote by \(\rdgGrMod{A}\) the same construction with \(G\)-graded right \(A\)-modules, which are equivalently left modules over the opposite ring, \(A^\opp\).
\end{definition}

\begin{remark}
  One should note that the closed morphisms are precisely the morphisms of complexes \(M \to N[p]\) and, in particular, the closed degree zero morphisms are precisely the usual morphisms of complexes.
\end{remark}

\begin{definition}
  To each \(G\)-graded \(k\)-algebra, \(A\), associate the category \(\A\) with objects the group \(G\) and morphisms given by
  \[\A(g_1, g_2) = A_{g_2 - g_1}\]
  and composition defined by the multiplication \(A_{g_2 - g_1}A_{g_3 - g_2} \subseteq A_{g_3 - g_1}\).

  We regard \(\A\) as a dg-category by considering the \(k\)-module of morphisms, \(\A(g_1, g_2)\), as the complex with \(A_{g_2 - g_1}\) in degree 0 and zero differential.
\end{definition}

\begin{lemma}\label{lem:GrModAsMod}
  Let \(G\) be an abelian group.
  If \(A\) is a \(G\)-graded algebra over \(k\) and \(\A\) the associated dg-category, then there is an isomorphism of dg-categories
  \[\ldgGrMod{A} \cong \dgMod{\A}.\]

  \begin{proof}
    We first construct a dg-functor \(F \colon \ldgGrMod{A} \to \dgMod{\A}\).
    For each \(g \in G\), denote by \(A(g)[0]\) the complex with \(A(g)\) in degree zero and consider the full subcategory of \(\ldgGrMod{A}\) of all such complexes.
    We see that a morphism \[f \in \ldgGrMod{A}(A(g)[0],M)^n\] is just the data of a morphism \(f^0 \colon A(g) \to M^n\) which gives
    \[\ldgGrMod{A}(A(g)[0],M)^n \cong \Gr{A}(A(g), M^n) \cong M^n_{-g}\]
    and hence \(M_{-g} := \ldgGrMod{A}(A(g)[0], M)\) is the complex with \(M^n_{-g}\) in degree \(n\).
    In particular, when \(M = A(h)[0]\), we have
    \[\ldgGrMod{A}(A(g)[0],A(h)[0]) := A(h)[0]_{-g} = \A(g,h),\]
    which allows us to identify this subcategory with \(\A\) via the enriched Yoneda embedding, \(A(h)[0]\) corresponding to the representable functor \(\A(-,h)\).
    Using this identification, we can define the image of \(M\) in \(\dgMod{\A}\) to be the dg-functor that takes an object \(g \in G\) to
    \[M_{-g} = \ldgGrMod{A}(A(g)[0], M)\]
    with structure morphism
    \[\A(g,h) \cong \ldgGrMod{A}(A(g)[0], A(h)[0]) \to \CH{k}(M_{-h}, M_{-g})\]
    induced by the representable functor \(\ldgGrMod{A}(-, M)\).
    The image of a morphism \(f \in \ldgGrMod{A}(M,N)\) is defined to be the natural transformation given by the collection of morphisms
    \[h^{A(-g)[0]}(f) \colon \ldgGrMod{A}(A(-g)[0], M) \to \ldgGrMod{A}(A(-g)[0],N)\]
    indexed by \(G\).

    Conversely, we note that the data of a functor \(M \colon \A^\opp \to \CH{k}\) is a collection of chain complexes, \(M_g := M(g)\), indexed by \(G\) and morphisms of complexes
    \[\begin{tikzcd}
    \cdots \arrow{r} & A_{g - h} \arrow{r}\arrow{d} & 0 \arrow{r}\arrow{d} & \cdots\\
    \cdots \arrow{r} & \CH{k}(M_g,M_h)^0 \arrow{r} & \CH{k}(M_g, M_h)^1 \arrow{r} & \cdots
    \end{tikzcd}\]    The non-zero arrow factors through \(Z^0(\CH{k}(M_g, M_h))\), so the structure morphism is equivalent to giving a morphism
    \[A_{g - h} \to \CH{k}(M_g,M_h)\]
    and thus \(M\) determines a complex of graded \(A\)-modules
    \[\widetilde{M} = \bigoplus_{g \in G}M_{-g}.\]
    A morphism \(\eta \colon M \to N\) is simply a collection of natural transformations \(\eta^p\) such that for each \(g \in G\) we have \(\eta^p(g) \in \CH{k}(M_g,N_g)^p\) and the naturality implies that \(\eta^p(g)\) is \(A\)-linear.
    The natural transformation \(\eta^p\) thus determines a morphism
    \[\bigoplus_{g \in G} \eta^p(-g) \in \ldgGrMod{A}\left(\widetilde{M}, \widetilde{N}\right)^p,\]
    and hence \(\eta\) determines a morphism in \(\ldgGrMod{A}\left(\widetilde{M}, \widetilde{N}\right)\), which is the collection of all such homogenous components.
    This defines a dg-functor \(\dgMod{\A} \to \ldgGrMod{A}\) which is clearly the inverse of \(F\).
  \end{proof}
\end{lemma}

\begin{remark}
  It is worth noting that it is natural from the ringoid perspective to reverse the weighting on the opposite ring in that, formally,
  \[A^\opp_{g} = \A^\opp(0,g) = \A(g,0) = A_{-g}\]
  so that \(\A^\opp(-, h) = \A(h,-)\) is the representable functor corresponding to the left module \(A^\opp(h)\) by
  \[\bigoplus_{g \in G} \A^\opp(-g,h) = \bigoplus_{g \in G} \A(h,-g) = \bigoplus_{g \in G} A_{-(g + h)} = \bigoplus_{g \in G} A^\opp_{g + h} = A^\opp(h).\]
  With this convention, when considering right modules, one can dispense with the formality of the opposite ring by constructing from a complex, \(M\), the dg-functor \(\A \to \CH{k}\) mapping \(g\) to \(M_g := \rdgGrMod{A}(A(-g)[0], M)\).
\end{remark}
When \(G = \Z^2\), and \(A\), \(B\) are \(\Z\)-graded algebras over \(k\), we denote the dg-category of chain complexes of \(G\)-graded \(B\)-\(A\)-bimodules by \(\ldgGrMod{A^\opp \otimes_k B}\).
We associate to the \(\Z^2\)-graded \(k\)-algebra \(A^\opp \otimes_k B\) the tensor product of the associated dg-categories, \(\A^\opp \otimes \B\).
Note that in the identification
\[\ldgGrMod{\left(A^\opp \otimes_k B\right)} \cong \dgMod{\A^\opp \otimes \B}\]
the weighting coming from the \(A\)-module structure is reversed, as in the remark above.

From this construction, we have a dg-enhancement, \(\hproj{\A}\), of the derived category of graded modules, \(\mathrm{D}(\Gr{A})\).
Passing through the machinery of Corollary~\ref{corollary: Toen}, we have an isomorphism in \(\Ho{\dgcat{k}}\)
\[\RHomc{\hproj{\A}, \hproj{\B}} \cong \hproj{\A^\opp \otimes \B}.\]
This allows us to identify an object, \(F\), of \(\RHomc{\hproj{A}, \hproj{B}}\) with a dg \(\A\)-\(\B\)-bimodule, \(P\), which in turn corresponds to a morphism \(\Phi_P \colon \A \to \hproj{\B}\) by way of the symmetric monoidal closed structure on \(\dgcat{k}\).

Following Section 3.3 of \cite{CS}, we identify the homotopy equivalence class, \([P]_{\op{Iso}}\), of \(P\) with \([\Phi_P] \in [\A, \hproj{\B}]\).
The extension of \(\Phi_P\),
\[P \otimes_\A - = \widehat{\Phi_P} \colon \hproj{\A} \to \hproj{\B}\]
descends to a morphism \([\widehat{\Phi_P}] \in [\hproj{\A}, \hproj{\B}]\)
and induces a triangulated functor that commutes with coproducts
\[\begin{tikzcd}[row sep=tiny]
H^0(\widehat{\Phi_P}) \colon \mathrm{D}(\Gr{A}) \arrow{r} & \mathrm{D}(\Gr{B})\\
\quad\quad M \arrow[mapsto]{r} & P \otimes_\A^\mathbf{L} M.
\end{tikzcd}\]
In particular, given an equivalence \(f \colon \mathrm{D}(\Gr{A}) \to \mathrm{D}(\Gr{B})\), we obtain from \cite{Lunts-Orlov} a quasi-equivalence
\[F \colon \hproj{\A} \to \hproj{\B}.\]
Tracing through the remarks above, we obtain an object \(P\) of \(\hproj{\A^\opp \otimes \B}\) providing an equivalence
\[H^0(\widehat{\Phi_P}) \colon \mathrm{D}(\Gr{A}) \to \mathrm{D}(\Gr{B}).\]

\section{Derived Morita Theory for Noncommutative Projective Schemes} \label{section: morita for NCP}

Let \(A\) and \(B\) left Noetherian connected graded \(k\)-algebras.
We want to extend the ideas from Section~\ref{subsection: graded Morita} to cover dg-enhancements of \(\mathrm{D}(\QGr{A})\). 

\subsection{Vanishing of a tensor product} \label{subsection: vanishing of tensor}

We recall a particularly nice type of property of objects in the setting of compactly generated triangulated categories.
In the sequel, many of our properties will be of this type, so we give this little gem a name.

\begin{definition} \label{definition: run the jewels}
  Let \(T\) be a compactly generated triangulated category.
  Let \(\mathtt{P}\) be a property of objects of \(T\). 
  We say that \(\mathtt{P}\) is \textbf{RTJ} if it satisifies the following three conditions.
  \begin{itemize}
  \item Whenever \(A \to B \to C\) is a triangle in \(T\) and \(\mathtt{P}\) holds for \(A\) and \(B\), then \(\mathtt{P}\) holds for \(C\). 
  \item If \(\mathtt{P}\) holds for \(A\), then \(\mathtt{P}\) holds for the translate \(A[1]\).
  \item Let \(I\) be a set and \(A_i\) be objects of \(T\) for each \(i \in I\). If \(\mathtt{P}\) holds for each \(A_i\), then \(\mathtt{P}\) holds for \(\bigoplus_{i \in I} A_i\). 
  \end{itemize}
\end{definition}

\begin{proposition} \label{proposition: RTJ properties}
  Let \(\mathtt{P}\) be an RTJ property that holds for a set of compact generators of \(T\). Then \(\mathtt{P}\) holds for all objects of \(T\).
\end{proposition}

\begin{proof}
  Let \(P\) be the full triangulated subcategory of objects for which \(\mathtt{P}\) holds. Then \(P\)
  \begin{itemize}
  \item contains a set of compact generators,
  \item is triangulated, and
  \item is closed under formation of coproducts. 
  \end{itemize}
  Thus, \(P\) is all of \(T\). 
\end{proof}

\begin{definition} \label{definition: tensor vanishing}
  Let \(M\) be a complex of left graded \(A\)-modules and let \(N\) be a complex of right graded \(A\)-modules. We say that the pair satisfies \(\bigstar(M,N)\) if the tensor product
  \begin{displaymath}
    \mathbf{R} \tau_{A^\opp} N \overset{\mathbf{L}}{\otimes}_{\mathcal A} \mathbf{R} Q_A M = 0
  \end{displaymath}
  vanishes. If \(\bigstar(M,N)\) holds for all \(M\) and \(N\), then we say that \(A\) satisfies \(\bigstar\). 
\end{definition}

\begin{proposition} \label{proposition: big star condition}
  Let \(A\) be a finitely generated, connected graded \(k\)-algebra.
  Assume that \(\mathbf{R}\tau_A\) and \(\mathbf{R}\tau_{A^\opp}\) commute with coproducts. Then \(A\) satisfies \(\bigstar\) if and only \(\bigstar( A(u), A(v) )\) holds for each \(u,v \in \Z\).
\end{proposition}

\begin{proof}
  The necessity is clear, so assume that \(\bigstar( A(u), A(v) )\) holds for each \(u,v \in \Z\). Note that \(\bigstar(M , A(v))\) holds for all \(v\) is an RTJ property of \(M\) that holds for a set of compact generators \(A(u), u \in \Z\). Thus, by Proposition~\ref{proposition: RTJ properties}, \(\bigstar(M , A(v))\) holds for all \(v\) holds for all \(M\) in \(\mathrm{D}(\Gr{A})\). Similarly, we can consider the property of \(N\) in \(\mathrm{D}(\Gr{(A^\opp)})\): \(\bigstar(M , N)\) holds for all objects \(M\) of \(\mathrm{D}(\Gr{A})\). This is also RTJ so \(\bigstar(M,N)\) holds for all \(M\) and \(N\).
\end{proof}

There are various types of projection formulas.
We record here two which will be useful in the sequel.

\begin{proposition} \label{proposition: projection formula}
  Let \(A\) be a finitely generated, connected graded \(k\)-algebra.
  Let \(P\) be a complex of bi-bi \(A\)-modules and let \(M\) be a complex of left graded \(A\)-modules. Assume \(\mathbf{R} \tau_A\) commutes with coproducts. There is natural quasi-isomorphism
  \begin{displaymath}
    ( \mathbf{R} \tau_A P ) \overset{\mathbf{L}}{\otimes}_{\mathcal A} M \to \mathbf{R} \tau_A \left( P \overset{\mathbf{L}}{\otimes}_{\mathcal A} M \right).
  \end{displaymath}
  Assume \(\mathbf{R} Q_A\) commutes with coproducts. There is natural quasi-isomorphism
  \begin{displaymath}
    ( \mathbf{R} Q_A P ) \overset{\mathbf{L}}{\otimes}_{\mathcal A} M \to \mathbf{R} Q_A \left( P \overset{\mathbf{L}}{\otimes}_{\mathcal A} M \right).
  \end{displaymath}
\end{proposition}

\begin{proof}
  We treat the \(\tau\) projection formula. The \(Q\) projection formula is analogous. By Corollary~\ref{corollary: Q preserves bimodules}, we see that the tensor product is well-defined. It suffices to exhibit a natural transformation for the underived functors applied to modules to generate the desired natural transformation. Given
  \begin{displaymath}
    \psi \otimes_{\mathcal A} m \in \GR{A}(A/A_{\geq m}, P) \otimes_{\mathcal A} M 
  \end{displaymath}
  we naturally get 
  \begin{align*}
    \widetilde{\psi} : A/A_{\geq m} & \to P \otimes_{\mathcal A} M \\
    a & \mapsto \psi(a) \otimes_{\mathcal A} m. 
  \end{align*}
  Taking the colimit gives the natural transformation. Let us look at the natural transformation in the case that \(P = A(u) \otimes_k A(v)\) and \(M = A(w)\). Recall that 
  \begin{gather*}
    \left( A(u) \otimes_k A(v) \right) \overset{\mathbf{L}}{\otimes}_{\mathcal A} A(w) \cong \left( A(u) \otimes_k A(v) \right) \otimes_{\mathcal A} A(w) \\ \cong \left( A(u) \otimes_k A(v) \right)_{\ast,w} \cong A(u) \otimes_k A_{v+w}. 
  \end{gather*}
  So
  \begin{displaymath}
    \mathbf{R}\tau_A \left( \left( A(u) \otimes_k A(v) \right) \overset{\mathbf{L}}{\otimes}_{\mathcal A} A(w) \right) \cong \mathbf{R} \tau_A \left( A(u) \otimes_k A_{v+w} \right)
  \end{displaymath}
  while 
  \begin{gather*}
    \mathbf{R}\tau_A \left( A(u) \otimes_k A(v) \right) \overset{\mathbf{L}}{\otimes}_{\mathcal A} A(w) \cong \mathbf{R}\tau_A \left( A(u) \otimes_k A(v) \right)_w \\ \cong \mathbf{R}\tau_A \left( (A(u) \otimes_k A(v))_{\ast,w} \right) \cong \mathbf{R}\tau_A \left( A(u) \otimes_k A_{v+w} \right)
  \end{gather*}
  which are compatible with the natural transformation. The property that the natural transformation is a quasi-isomorphism is RTJ in each entry. Thus, it holds for all \(P\) and \(M\) by Proposition~\ref{proposition: RTJ properties}. 
\end{proof}

For the hypothesis, recall Definition~\ref{definition: delightful couple}. 

\begin{proposition} \label{proposition: vanishing of tensor}
  Assume \(A\) is delightful. Then \(\bigstar\) holds for \(A\).
\end{proposition}

\begin{proof}
  By Proposition~\ref{proposition: big star condition}, it suffices to check \(\bigstar(M,A(v))\) for each \(v\). This is equivalent to \(\bigstar(M,\bigoplus_v A(v))\). Equipping the sum with a bi-bi structure as \(\Delta\), we reduce to checking \(\bigstar(M,\Delta)\). Using Proposition~\ref{proposition: when beta is an isomorphism} and Lemma~\ref{lemma: exact triangles} for \(A\) and \(A^\opp\), we have a natural quasi-isomorphism
  \[\mathbf{R}\tau_{A^\opp} \Delta \overset{\mathbf{L}}{\otimes}_{\mathcal A} \mathbf{R}Q_A M \cong \mathbf{R}\tau_{A} \Delta \overset{\mathbf{L}}{\otimes}_{\mathcal A} \mathbf{R}Q_{A} M.\]
  Using Proposition~\ref{proposition: projection formula}, we have a natural quasi-isomorphism
  \[\mathbf{R}\tau_{A} \Delta \overset{\mathbf{L}}{\otimes}_{\mathcal A} \mathbf{R}Q_{A} M \cong \mathbf{R}\tau_{A} \left( \Delta \overset{\mathbf{L}}{\otimes}_{\mathcal A} \mathbf{R}Q_{A} M  \right) \cong \mathbf{R}\tau_{A} \left( \mathbf{R}Q_{A} M \right) = 0.\]
\end{proof}

\subsection{Duality}

One can regard the bimodule \(\mathbf{R}Q_{A \otimes_k A^\opp}\Delta\) as a sum of \(A\)-modules
\[\mathbf{R}Q_{A \otimes_k A^\opp}\Delta = \bigoplus_{x} (\mathbf{R}Q_{A \otimes_k A^\opp}\Delta)_{\ast,x}\]
and define for any object, \(M\), of \(\ldgGrMod{A}\) the object
\[\mathbf{R}\underline{\op{Hom}}_A(M, \mathbf{R}Q_{A \otimes_k A^\opp}\Delta) = \bigoplus_x \mathbf{R}\op{Hom}_A(M, \mathbf{R}Q_{A \otimes_k A^\opp}\Delta_{\ast,x})\]
of \(\rdgGrMod{A}\).
Consider the functor
\begin{align*}
  (-)^{\vee} : \ldgGrMod{A}^\opp & \to \rdgGrMod{A} \\
  M & \mapsto \mathbf{R}\underline{\op{Hom}}_A \left( M , \mathbf{R}Q_{A \otimes_k A^\opp} \Delta \right)
\end{align*}

\begin{lemma} \label{lemma: duality really is a duality}
  Assume \(A\) is delightful. Then the natural map 
  \begin{displaymath}
    \id \to (-)^{\vee \vee}
  \end{displaymath}
  given by evaluation is a quasi-isomorphism for \(\mathbf{R}Q_A A(x)\), for all \(x\). Furthermore, there are quasi-isomorphisms
  \begin{displaymath}
    \left( \mathbf{R}Q_A A(x) \right)^\vee \cong \mathbf{R}Q_{A^\opp} A(-x).
  \end{displaymath}
\end{lemma}

\begin{proof}
  We first exhibit the latter quasi-isomorphisms. We have 
  \begin{displaymath}
    (\mathbf{R}Q_A A(x))^\vee \cong \mathbf{R}\op{Hom}_A (A(x), \mathbf{R}Q_{A \otimes_k A^\opp} \Delta) \cong (\mathbf{R}Q_{A \otimes_k A^\opp} \Delta)_{-x,\ast}
  \end{displaymath}
  since \(\mathbf{R}Q_{A \otimes_k A^\opp} \Delta\) is right orthogonal to \(\tau_A\)-torsion. From Proposition~\ref{proposition: when beta is an isomorphism}, we get 
  \begin{displaymath}
    (\mathbf{R}Q_{A \otimes_k A^\opp} \Delta)_{-x,\ast} \cong (\mathbf{R}Q_{A^\opp} \Delta)_{-x,\ast} = \mathbf{R}Q_{A^\opp} A(-x).
  \end{displaymath}
  Applying this twice, we get 
  \begin{displaymath}
    (\mathbf{R}Q_A A(x))^{\vee \vee} \cong \mathbf{R}Q_A A(x).
  \end{displaymath}
  We just need to check that the natural map \(\nu : 1 \to (-)^{\vee \vee}\) induces the identity after this quasi-isomorphism.

  Note that we found a map
  \begin{displaymath}
    A(-x) \to \mathbf{R}Q_A A(-x) \to ( \mathbf{R}Q_A A(x) )^\vee 
  \end{displaymath}
  inducing the quasi-isomorphism \((\mathbf{R}Q_A A(x))^{\vee \vee} \cong \mathbf{R}Q_A A(x)\). One can identify the image of \(1\) as a map 
  \begin{displaymath}
    \mathbf{R}Q_A A(x) \to \mathbf{R}Q_{A \otimes_k A^\opp} \Delta (0,x) \cong \mathbf{R}Q_A ( \Delta (x,0) )
  \end{displaymath}
  which, after applying the quasi-isomorphism, is the natural inclusion. Evaluating this at \(a \in \mathbf{R}Q_A A(x)\) gives \(a\) back. Thus, we see that \(\nu\) is quasi-fully faithful on \(\mathbf{R}Q_A A(x)\) for all \(x\). 
\end{proof}

\begin{definition}
  Let \(Q\A\) be the full dg-subcategory of \(C(\Gr{A})\) with objects given by \(Q_A\) applied to injective resolutions of \(A(x)\) for all \(x\). 
\end{definition}

\begin{corollary} \label{corollary: duality is a duality}
  Assume that \(A\) is delightful. The functor \((-)^\vee\) induces a quasi-equivalence \((Q\A)^\opp \cong Q(\A^\opp)\). 
\end{corollary}

\begin{proof}
  From Lemma~\ref{lemma: duality really is a duality}, we see that \((-)^\vee\) is quasi-fully faithful on \(Q\A\) and has quasi-essential image \(Q(\A^\opp)\). 
\end{proof}

\begin{lemma} \label{lemma: trace map}
  Assume that \(A\) is delightful. There is a natural map 
  \begin{displaymath}
    \eta : M^\vee \overset{\mathbf{L}}{\otimes}_{\mathcal A} N \to \op{Hom}_A(M,N).
  \end{displaymath}
  Then \(\eta\) is a quasi-isomorphism for any \(M\) and any \(N \cong \mathbf{R}Q_A N\). 
\end{lemma}

\begin{proof}
  First, note that we have the natural map 
  \begin{displaymath}
    M^\vee \overset{\mathbf{L}}{\otimes}_{\mathcal A} N \to \mathbf{R}\underline{\op{Hom}}_A(M, \mathbf{R}Q_{A \otimes_k A^\opp} \Delta \overset{\mathbf{L}}{\otimes}_{\mathcal A} N).
  \end{displaymath}
  For \(M = A(x)\), we see this map is a quasi-isomorphism using the fact that \(A\) satisfies \(\bigstar\) from Proposition~\ref{proposition: vanishing of tensor}. 
  Since \(A\) satisfies \(\bigstar\), the map \(N \cong \Delta \otimes_{\mathcal A} N \to \mathbf{R}Q_{A \otimes_k A^\opp} \Delta \overset{\mathbf{L}}{\otimes}_{\mathcal A} N\) is a quasi-isomorphism. So the map  
  \begin{displaymath}
    \mathbf{R}\underline{\op{Hom}}_A(M, N) \to \mathbf{R}\underline{\op{Hom}}_A(M, \mathbf{R}Q_{A \otimes_k A^\opp} \Delta \overset{\mathbf{L}}{\otimes}_{\mathcal A} N)
  \end{displaymath}
  is also a quasi-isomorphism. Combining the two gives the desired quasi-isomorphism for \(M = A(x)\). But the condition \(\eta\) is a quasi-isomorphism is RTJ in \(M\) so is true for all \(M\) by Proposition~\ref{proposition: RTJ properties}
\end{proof}

\subsection{Products} \label{section: products}

\begin{definition}
  For a finitely generated, connected graded \(k\)-algebra, \(A\), let \(\hinj{\Gr{A}}\) be the full dg-subcategory of \(C(\Gr{A})\) with objects the K-injective complexes of Spaltenstein \cite{Spaltenstein}. Similarly, we let \(\hinj{\QGr{A}}\) be the full dg-subcategory of \(C(\QGr{A})\) with objects the K-injective complexes.
\end{definition}

\begin{lemma} \label{lemma: omega of Kinj is Kinj}
  The functor 
  \begin{displaymath}
    \omega : \hinj{\QGr{A}} \to \hinj{\Gr{A}}
  \end{displaymath}
  is well-defined. Moreover, \(H^0(\omega)\) is an equivalence with its essential image. 
\end{lemma}

\begin{proof}
  For the first statement, we just need to check that \(\omega\) takes K-injective complexes to K-injective complexes. 
  This is clear from the fact that \(\omega\) is right adjoint to \(\pi\), which is exact.
  
  To see this is fully faithful, we recall that \(\pi \omega \cong \op{Id}\) so
  \[\hinj{\Gr{A}}(\omega M, \omega N) \cong \hinj{\QGr{A}}(\pi\omega M, N) \cong \hinj{\QGr{A}}(M,N).\]
\end{proof}

\begin{remark} \label{remark: enhancement of DQGr}
  Using Lemma~\ref{lemma: omega of Kinj is Kinj}, we can either use \(\hinj{\QGr{A}}\) or its image under \(\omega\) in \(\hinj{\Gr{A}}\) as an enhancement of \(\mathrm{D}(\QGr{A})\). 
\end{remark}

Consider the full dg-subcategory of \(\hinj{\QGr{A \otimes_k B}}\) consisting of the objects
\[\pi_{A \otimes_k B} (A(u) \otimes_k B(v))\]
for all \(u,v\). Denote this subcategory by \(\mathcal E\).

\begin{lemma} \label{lemma: another model for QA otimes QB}
  If \(A\) and \(B\) are both Ext-finite, left Noetherian, and right Noetherian, then the dg-category \(\mathcal E\) is naturally quasi-equivalent to \(Q \A \otimes_k Q \B\).
\end{lemma}

\begin{proof}
  Recall that \(Q \A\) is the full dg-subcategory of \(\ldgGrMod{A}\) consisting of \(Q_{A}\) applied to injective resolutions of \(A(u)\), loosely denoted by \(\mathbf{R}Q_A A(u)\), and similarly for \(Q \B\). We have the exact functor
  \begin{displaymath}
    \otimes_k : \ldgGrMod{A} \otimes_k \ldgGrMod{B} \to \ldgGrMod{A \otimes_k B}
  \end{displaymath}
  which tensors a pair of modules over \(k\) to yield a bimodule. First consider the triangle
  \begin{gather*}
    \mathbf{R}\tau_{A \otimes B} (\mathbf{R} Q_A A(u) \otimes_k \mathbf{R} Q_B B(v) ) \to \mathbf{R} Q_A A(u) \otimes_k \mathbf{R} Q_B B(v) \\ \to \mathbf{R}Q_{A \otimes B} (\mathbf{R} Q_A A(u) \otimes_k \mathbf{R} Q_B B(v) ).
  \end{gather*}
  By Proposition~\ref{proposition: bi-torsion is a composition}, we have 
  \begin{displaymath}
    \mathbf{R}Q_{A \otimes B} (\mathbf{R} Q_A A(u) \otimes_k \mathbf{R} Q_B B(v) ) \cong \mathbf{R} Q_A \left( \mathbf{R}Q_B \left( \mathbf{R} Q_A A(u) \otimes_k \mathbf{R} Q_B B(v) \right) \right).
  \end{displaymath}
  Since \(\mathbf{R}\tau_B\) commutes with coproducts, we have a natural quasi-isomorphism
  \begin{gather*}
    \mathbf{R} Q_A \left( \mathbf{R}Q_B \left( \mathbf{R} Q_A A(u) \otimes_k \mathbf{R} Q_B B(v) \right) \right) \cong \mathbf{R}Q_A^2 A(u) \otimes_k \mathbf{R}Q_B^2 B(v) \\ \cong \mathbf{R}Q_A A(u) \otimes_k \mathbf{R}Q_B B(v). 
  \end{gather*}
  Thus, 
  \begin{displaymath}
    \mathbf{R} Q_A A(u) \otimes_k \mathbf{R} Q_B B(v) \to \mathbf{R}Q_{A \otimes B} (\mathbf{R} Q_A A(u) \otimes_k \mathbf{R} Q_B B(v) )
  \end{displaymath}
  is quasi-isomorphism for all \(u,v\) with \(\tau_{A \otimes_k B}\) torsion cone. The same consideration shows that the map 
  \begin{displaymath}
    A(u) \otimes_k A(v) \to \mathbf{R} Q_A A(u) \otimes_k \mathbf{R} Q_B B(v)
  \end{displaymath}
  induces a quasi-isomorphism
  \begin{displaymath}
    \mathbf{R}Q_{A \otimes B}( A(u) \otimes_k B(v) ) \to \mathbf{R}Q_{A \otimes B} (\mathbf{R} Q_A A(u) \otimes_k \mathbf{R} Q_B B(v) )
  \end{displaymath}
  with \(\tau_{A \otimes_k B}\) torsion kernel. Now we check that these morphisms induce quasi-isomorphisms on the morphism spaces giving our desired quasi-equivalence. We have a commutative diagram
  \begin{center}
    \scalebox{0.8}{
      \begin{tikzpicture}[scale=.6,level/.style={->,>=stealth,thick}]
	\node (a) at (-5,5) {\(\op{Hom}(\mathbf{R} Q_A A(u) \otimes_k \mathbf{R} Q_B B(v),\mathbf{R} Q_A A(x) \otimes_k \mathbf{R} Q_B B(y))\)};
	\node (b) at (6,2.5) {\(\op{Hom}(A(u) \otimes_k B(v),\mathbf{R} Q_A A(x) \otimes_k \mathbf{R} Q_B B(y))\)};
	\node (c) at (-7,0) {\(\op{Hom}(\mathbf{R} Q_A A(u) \otimes_k \mathbf{R} Q_B B(v),\mathbf{R}Q_{A \otimes_k B} (\mathbf{R} Q_A A(x) \otimes_k \mathbf{R} Q_B B(y)))\)};
	\node (d) at (5,-2.5) {\(\op{Hom}(A(u) \otimes_k B(v),\mathbf{R}Q_{A \otimes_k B} (\mathbf{R} Q_A A(x) \otimes_k \mathbf{R} Q_B B(y)))\)};
	\node (e) at (-5,-5) {\(\op{Hom}(\mathbf{R}Q_{A \otimes_k B} (\mathbf{R} Q_A A(u) \otimes_k \mathbf{R} Q_B B(v)),\mathbf{R}Q_{A \otimes_k B} (\mathbf{R} Q_A A(x) \otimes_k \mathbf{R} Q_B B(y)))\)};
	\draw[level] (a) -- node[left] {\(a\)} (c) ;
	\draw[level] (e) -- node[left] {\(b\)} (c) ;
	\draw[level] (a) -- node[above] {\(c\)} (b) ;
	\draw[level] (b) -- node[right] {\(d\)} (d) ;
	\draw[level] (c) -- node[below] {\(e\)} (d) ;
    \end{tikzpicture} }
  \end{center}
  and we want to know first that \(a\) and \(b\) are quasi-isomorphisms. We know that \(b\) is a quasi-isomorphism since \(\mathbf{R}\tau_{A \otimes_k B}\) is left orthogonal to \(\mathbf{R}Q_{A \otimes_k B}\) so we only need to check \(a\). Since \(A(u) \otimes_k B(v)\) is free and 
  \begin{displaymath}
    \mathbf{R} Q_A A(u) \otimes_k \mathbf{R} Q_B B(v) \to \mathbf{R}Q_{A \otimes B} (\mathbf{R} Q_A A(u) \otimes_k \mathbf{R} Q_B B(v) )
  \end{displaymath}
  is a quasi-isomorphism, \(d\) is a quasi-isomorphism. Since \(\mathbf{R}Q_A\) and \(\mathbf{R}Q_B\) commute with coproducts, using tensor-Hom adjunction shows that \(c\) is a quasi-isomorphism. Finally, since the cone over the map
  \begin{displaymath}
    A(u) \otimes_k A(v) \to \mathbf{R} Q_A A(u) \otimes_k \mathbf{R} Q_B B(v)
  \end{displaymath}
  is annihilated by \(\tau_{A \otimes_k B}\), we see that \(e\) is also a quasi-isomorphism. This implies that \(a\) is a quasi-isomorphism.
  By an analogous argument, the endomorphisms of \(\mathbf{R}Q_{A \otimes B} (A(u) \otimes_k B(v))\) and \(\mathbf{R}Q_{A \otimes B} (\mathbf{R} Q_A A(u) \otimes_k \mathbf{R} Q_B B(v) )\) are quasi-isomorphic.
\end{proof}

\subsection{The quasi-equivalence}

Now we turn to the main result. 

\begin{theorem} \label{theorem: derived morita for NCP}
  Let \(k\) be a field. Let \(A\) and \(B\) be connected graded \(k\)-algebras. If \(A\) and \(B\) form a delightful couple, then there is a natural quasi-equivalence 
  \begin{displaymath}
    F : \hinj{\QGr{A^\opp \otimes_k B}} \to \RHomc{ \hinj{\QGr{A}}, \hinj{\QGr{B}} }
  \end{displaymath}
  such that for an object \(P\) of \(\mathrm{D}(\QGr{A^\opp \otimes_k B})\), the exact functor \(H^0(F(P))\) is isomorphic to 
  \begin{displaymath}
    \Phi_P(M) :=  \pi_B \left( \mathbf{R}\omega_{A^\opp \otimes_k B} P \overset{\mathbf{L}}{\otimes}_{\mathcal A} \mathbf{R}\omega_A M \right).
  \end{displaymath}
\end{theorem}

\begin{proof}
  Applying Corollary~\ref{corollary: Toen}, it suffices to provide a quasi-equivalence
  \begin{displaymath}
    G : \hinj{\QGr{A^\opp \otimes_k B}} \to \hproj{ (Q \mathcal A)^\opp \otimes Q \mathcal B}
  \end{displaymath}
  Using Corollary~\ref{corollary: duality is a duality}, we have a quasi-equivalence
  \begin{displaymath}
    \hproj{ (Q \mathcal A)^\opp \otimes Q \mathcal B} \cong \hproj{ Q \mathcal A^\opp \otimes Q \mathcal B}. 
  \end{displaymath}
  From Lemma~\ref{lemma: another model for QA otimes QB} we have a quasi-fully faithful functor 
  \begin{displaymath}
    Q \mathcal A^\opp \otimes Q \mathcal B \to \hinj{\QGr{A^\opp \otimes_k B}}. 
  \end{displaymath}
  This gives a functor 
  \begin{displaymath}
    \hinj{\QGr{A^\opp \otimes_k B}} \to \hproj{Q \mathcal A^\opp \otimes Q \mathcal B}
  \end{displaymath}
  which is a quasi-equivalence by a standard argument, see e.g. \cite[Theorem 5.1]{Dyckerhoff}.
  
  Tracing out the quasi-equivalences, one just needs to manipulate 
  \begin{align*}
    \op{Hom} (\mathbf{R}Q_A A(x)^\vee \otimes_k \mathbf{R}Q_B B(y), P) & \cong \op{Hom} ( \mathbf{R}Q_B B(y) , \op{Hom}( \mathbf{R}Q_A A(x)^\vee, \mathbf{R}\omega_{A^\opp \otimes_k B} P)) \\
    & \cong \op{Hom} ( \mathbf{R}Q_B B(y) , \mathbf{R}\omega_{A^\opp \otimes_k B} P \overset{\mathbf{L}}{\otimes}_{\mathcal A} \mathbf{R}Q_A A(x) ) 
  \end{align*}
  using Propostion~\ref{proposition: vanishing of tensor} and Lemma~\ref{lemma: trace map}. This says that the induced continuous functor is
  \begin{displaymath}
    M \mapsto \pi_B \left( \mathbf{R}\omega_{A^\opp \otimes_k B} P \overset{\mathbf{L}}{\otimes}_{\mathcal A} \mathbf{R}\omega_A M \right). 
  \end{displaymath}
\end{proof}

The following statement is now a simple application of Theorem~\ref{theorem: derived morita for NCP} and results of \cite{Lunts-Orlov}. 

\begin{corollary} \label{corollary: NCP morita}
  Let \(A\) and \(B\) be a delightful couple of connected graded \(k\)-algebras with \(k\) a field. Assume that there exists an equivalence
  \begin{displaymath}
    f : \mathrm{D} (\QGr{A}) \to \mathrm{D} (\QGr{B}).
  \end{displaymath}
  Then there exists an object \(P \in D ( \QGr{A^\opp \otimes_k B} )\) such that 
  \begin{displaymath}
    \Phi_P : \mathrm{D} ( \QGr{A}) \to \mathrm{D} (\QGr{B} )
  \end{displaymath}
  is an equivalence.
\end{corollary}

\begin{proof}
  Applying \cite[Theorem 1]{Lunts-Orlov} we know there is a quasi-equivalence between the unique enhancements, i.e. there is an \( F \in [ \hinj{ \QGr{A}}, \hinj{ \QGr{B}} ]\) giving an equivalence
  \begin{displaymath}
    H^0(F) : H^0(\hinj{ \QGr{A} }) = \mathrm{D}(\QGr{A}) \to H^0(\hinj{ \QGr{B} }) = \mathrm{D}(\QGr{B}).
  \end{displaymath}
  Then, by Theorem~\ref{theorem: derived morita for NCP}, there exists a \(P \in \mathrm{D}(\QGr{A^\opp \otimes_k B})\) such that \(\Phi_P = H^0(F)\). 
\end{proof}

We wish to identify the kernels as objects of the derived category of an honest noncommutative projective scheme.
In general, one can only hope that kernels obtained as above are objects of the derived category of a noncommutative (bi)projective scheme.
However, we have the following special case in which we can collapse the \(\Z^2\)-grading to a \(\Z\)-grading.

\begin{corollary} \label{corollary: NCP morita degree 1}
  Let \(A\) and \(B\) be a delightful couple of connected graded \(k\)-algebras with \(k\) a field that are both generated in degree one.
  Assume that there exists an equivalence
  \begin{displaymath}
    f : \mathrm{D} (\QGr{A}) \to \mathrm{D} (\QGr{B}).
  \end{displaymath}
  Then there exists an object \(P \in D ( \QGr{A^\opp \times_k B} )\) that induces an equivalence
  \[\begin{tikzcd}[row sep=tiny]
  \mathrm{D}(\QGr{A}) \arrow{r}& \mathrm{D}(\QGr{B})\\
  M \arrow[mapsto]{r} & \pi_B\left(\mathbb{V}(P) \otimes^\mathbf{L} \mathbf{R}\omega_A M\right)
  \end{tikzcd}\]
\end{corollary}

\begin{proof}
  The equivalence \(\mathbb{V}\) of Theorem~\ref{theorem: Van Rompay} extends naturally to a quasi-equivalence
  \[\mathbb{V} \colon \hinj{\QGr{S}} \to \hinj{\QGr{T}}.\]
  Now choose \(P\) such that \(\mathbb{V}(P)\) is homotopy equivalent to the kernel obtained by an application of Corollary~\ref{corollary: NCP morita}, so the desired equivalence is \(\Phi_{\mathbb{V}(P)}\).
\end{proof}

Coming back to Example~\ref{example: ncCY}. We ask the following question. 

\begin{question}
  Fix \(q_{ij} \in \mathbf{C}\). Then two noncommutative projective schemes \(A_q^\phi\) and \(A_q^{\phi^\prime}\) are isomorphic if and only if they are derived equivalent. 
\end{question}

In the commutative case, this is a derived Torelli statement which one can understand via matrix factorizations \cite{OrlovMF} and the Mather-Yau theorem \cite{MatherYau}.  

\bibliographystyle{amsalpha}
\bibliography{biblio}
\end{document}